\documentclass[11pt,a4paper]{scrartcl}
\usepackage[utf8]{inputenc}
\usepackage[english]{babel}
\usepackage{helvet, stmaryrd}
\usepackage{amsmath}
\usepackage{amsthm}
\usepackage{amsfonts}
\usepackage{amssymb}
\usepackage{mathtools}
\usepackage{mathrsfs}
\usepackage{dsfont}
\usepackage{tabularx}
\usepackage{xcolor}
\usepackage{bbm}
\usepackage[paper=a4paper,left=25mm,right=25mm,top=25mm,bottom=35mm]{geometry}

\usepackage{hyperref}
\usepackage[noabbrev,capitalize]{cleveref}

\author{Helena Kremp}
\title{Multidimensional SDE with distributional drift and Levy noise}

\newtheorem{theorem}{Theorem}[section]
\newtheorem{definition}[theorem]{Definition}     
\newtheorem{proposition}[theorem]{Proposition}
\newtheorem{lemma}[theorem]{Lemma}	

\newtheorem{corollary}[theorem]{Corollary}

\newtheorem{remark}[theorem]{Remark}
\newtheorem{assumption}[theorem]{Assumption}

\numberwithin{equation}{section}

\newcommand{\R}{\mathbb{R}}
\newcommand{\Q}{\mathbb{Q}}
\newcommand{\N}{\mathbb{N}}
\newcommand{\E}{\mathds{E}}
\newcommand{\p}{\mathds{P}}
\newcommand{\1}{\mathds{1}}
\newcommand{\F}{\mathcal{F}}
\newcommand{\La}{\mathcal{L}^{\alpha}_{\nu}}
\let\oldmathcal\mathcal
\newcommand{\calC}{\mathscr{C}}
\newcommand{\CTcalC}{C_T\calC}
\newcommand{\calK}{\oldmathcal{K}}
\newcommand{\calX}{\mathscr{X}}

\newcommand{\calD}{\mathscr{D}_{\overline{T},T}^{\theta}}
\newcommand{\calG}{\mathscr{G}}
\renewcommand{\mathcal}[1]{\mathscr{#1}}
\renewcommand{\epsilon}{\varepsilon}

\newcommand{\para}{\varolessthan}
\newcommand{\arap}{\varogreaterthan}
\newcommand{\reso}{\varodot}

\newcommand{\squeeze}[2][0]{\mbox{$\medmuskip=#1mu\displaystyle#2$}}

\DeclarePairedDelimiter{\abs}{\lvert}{\rvert}
\DeclarePairedDelimiter{\norm}{\lVert}{\rVert}

\DeclarePairedDelimiter{\brackets}{[}{]}
\DeclarePairedDelimiter{\paren}{(}{)}

\begin{document}
\begin{center}
\begin{huge}
Multidimensional SDE with distributional drift and Lévy noise\\
\end{huge}
\begin{Large}
Helena Kremp, Nicolas Perkowski\\
\end{Large}
Freie Universität Berlin
\vspace{1cm}
\hrule
\end{center}
\begin{section}*{Abstract}
We solve multidimensional SDEs with distributional drift driven by symmetric, $\alpha$-stable Lévy processes for $\alpha\in (1,2]$ by studying the associated (singular) martingale problem and by solving the Kolmogorov backward equation. We allow for drifts of regularity $(2-2\alpha)/3$, and in particular we go beyond the by now well understood ``Young regime'', where the drift must have better regularity than $(1-\alpha)/2$. The analysis of the Kolmogorov backward equation in the low regularity regime is based on paracontrolled distributions. As an application of our results we construct a Brox diffusion with L\'evy noise.
\end{section}
\vspace{0,5cm}
\hrule
\vspace{1cm}

\begin{section}{Introduction}
We study the weak well-posedness of L\'evy-driven stochastic differential equations with distributional drift,
\begin{equation}\label{eq:intro-sde}
dX_{t}=V(t,X_{t})dt + dL_{t},\qquad X_{0}=x\in\R^{d},
\end{equation} 
where $L$ is a non-degenerate, symmetric, $\alpha$-stable Lévy process for $\alpha\in (1,2]$, and $V(t,\cdot)$ is a distribution in the space variable for each $t\geqslant 0$.

The special case where $L$ is a Brownian motion has received lots of attention in recent years, since such \emph{singular diffusions} arise as models for stochastic processes in random media. For example, as random directed polymers~\cite{Alberts2014, Delarue2016, Caravenna2017}, self-attracting Brownian motion in a random medium~\cite{Cannizzaro2018}, or as continuum analogue of Sinai's random walk in random environment (Brox diffusion,~\cite{Brox1986}). Singular diffusions also arise as ``stochastic characteristics'' of singular SPDEs, for example the KPZ equation~\cite{Gubinelli2017KPZ} or the parabolic Anderson model~\cite{Cannizzaro2018}.

Of course, for distributional $V$ the point evaluation $V(t,X_t)$ is not meaningful, so a priori it is not clear how to even make sense of~\eqref{eq:intro-sde}. The right perspective is not to consider $V(t,X_t)$ at a fixed time, but rather to work with the integral $\int_0^t V(s,X_s) ds$. Because of small scale oscillations of $X$, which are induced by the oscillations of $L$, we only ``see an averaged version'' of $V$ and this gives rise to some regularization. At least for a Brownian motion or for a sufficiently ``wild'' L\'evy jump process, on the other hand we would not expect any regularization from a Poisson process. In the Brownian case, this intuition can be made rigorous in different ways, for example via a Zvonkin transform which removes the drift~\cite{Zvonkin1974, Veretennikov1981, Bass2001, Krylov2005, Flandoli2010, Flandoli2017}, by considering the associated martingale problem and by constructing a domain for the singular generator~\cite{Flandoli2003, Delarue2016, Cannizzaro2018}, or by Dirichlet forms~\cite{Mathieu1994}. In the one-dimensional case it is also possible to apply an It\^o-McKean construction based on space and time transformations~\cite{Brox1986}.

Here we follow the martingale problem approach, in the spirit of~\cite{Delarue2016, Cannizzaro2018} who considered the Brownian case. Formally, $X$ solves~\eqref{eq:intro-sde} if and only if it solves the martingale problem for the generator $\mathcal G^V = \partial_t - (-\Delta)^{\alpha/2} + V\cdot \nabla$, where the fractional Laplacian is the generator of $L$ (and later we will consider slightly more general $L$). That is, for all functions $u$ in the domain of $\mathcal G^V$, the process $u(t,X_t) - u(0,x) - \int_0^t \mathcal G^V u(s,X_s) ds$, $t \geqslant 0$, is a martingale. The difficulty is that the domain of $\mathcal G^V$ necessarily has trivial intersection with the smooth functions: If $u$ is smooth, then $(\partial_t - (-\Delta)^{\alpha/2})u$ is smooth as well, while for non-constant $u$ the product $V \cdot \nabla u$ is only a distribution and not a continuous function. If we want $\mathcal G^V u$ to be a continuous function, then $u$ has to be non-smooth so that $(\partial_t - (-\Delta)^{\alpha/2})u$ is also a distribution which has appropriate cancellations with $V\cdot \nabla u$ and the sum of these terms is a continuous functions.

We can find such $u$ by solving the Kolmogorov backward equation
\begin{align}\label{eq:pde}
\partial_{t}u=(-\Delta)^{\alpha/2} u-V\cdot\nabla u+f,\quad u(T,\cdot)=u^{T},
\end{align}
for suitable continuous functions $f$ and $u^T$, so that $\mathcal G^V u = f$ by construction. Given $V \in C([0,T], \mathcal C^\beta(\R^d, \R^d))$, where $\mathcal C^\beta = B_{\infty,\infty}^\beta$, the regularization obtained from $-(-\Delta)^{\alpha/2}$ suggests that $u(t,\cdot) \in \mathcal C^{\alpha+\beta}$. Therefore, $\nabla u(t,\cdot) \in \mathcal C^{\alpha+\beta-1}$, and since the product $V(t,\cdot) \cdot \nabla u(t,\cdot)$ is well posed if and only if the sum of the regularities of the factors is strictly positive, we need $\alpha+2\beta-1 >0$, or $\beta > (1-\alpha)/2$. We call this the \emph{Young regime}, in analogy with the regularity requirements that are needed for the construction of the Young integral.

There have been several results on singular L\'evy SDEs in the Young regime in recent years. Athreya, Butkovsky and Mytnik~\cite{Athreya2018} consider the time-homogeneous one-dimensional case and construct weak solutions via a Zvonkin transform, before establishing strong uniqueness and existence by a Yamada-Watanabe type argument (which in particular is restricted to $d=1$). Two nearly simultaneous works Ling and Zhao~\cite{Ling2019} respectively de Raynal and Menozzi~\cite{deRaynal2019} consider the multi-dimensional (time-homogeneous resp. time-inhomogeneous) case and they prove existence and uniqueness for the martingale problem. They even allow $V \in C([0,T], B_{p,q}^\beta)$ for general $p,q$ (under suitable conditions), but they are still restricted to the Young regime and for $p=q=\infty$ they require $\beta > (1-\alpha)/2$. Let us also mention~\cite{Harang2020Regularity} who prove pathwise regularization by noise results for SDEs driven by (very irregular) fractional L\'evy noise, based on the methods of~\cite{Catellier2016, Harang2020Cinfinity}.

Here we go beyond the Young regime and we treat the multi-dimensional time-inhomogeneous case. On the other hand, we only work with $B_{\infty,\infty}$ Besov spaces and not with $B_{p,q}$ for general $p,q$. To go beyond the Young regime we use techniques from singular SPDEs. More precisely, following the ideas of~\cite{Delarue2016, Cannizzaro2018} in the Brownian case, we use paracontrolled distributions~\cite{Gubinelli2015Paracontrolled} to solve~\eqref{eq:pde} for $\beta > (2-2\alpha)/3$ and $\alpha \in (1,2]$. The idea is to treat $u$ as a perturbation of the linearized equation with additive noise, $\partial_t w = (-\Delta)^{\alpha/2} w - V$, and to leverage this to gain some regularity. This works as long as the nonlinearity $V \cdot \nabla u$ is of lower order than the linear operator $(-\Delta)^{\alpha/2}$, i.e. if $\alpha > 1$. And indeed in that case we have $(2-2\alpha)/3 < (1-\alpha)/2$, and we can go beyond the Young regime.

Being able to go beyond the Young regime is important for our main application, the construction of a ``Brox jump diffusion'' with $\alpha$-stable L\'evy noise. Here  $d=1$ and $V$ is a (periodic) space white noise, so in particular we can only take $\beta = -1/2-\varepsilon$ for $\varepsilon>0$, which is never in the Young regime, not even in the Brownian case $\alpha=2$. We also indicate how to adapt our constructions in order to treat a non-periodic white noise. On the other hand, we do not study the qualitative behavior of the solution and we leave this for future research.

\paragraph{Structure of the paper} In \cref{sect:pre} we collect some background material on Besov spaces and $\alpha$-stable L\'evy processes, and we discuss the Schauder estimates for the fractional Laplacian. In \cref{sect:pde} we then solve the Kolmogorov backward equation. Our main theorem concerning the existence and uniqueness of a solution to the martingale problem is proven in \cref{sect:mthm},
while in \cref{sect:bd} we construct the Brox diffusion with Lévy noise.

\end{section}

\begin{section}{Preliminaries}\label{sect:pre}

In this section, we introduce some technical ingredients that we will need in the sequel.

Let $(p_{j})_{j\geqslant -1}$ be a smooth dyadic partition of unity, i.e. a family of functions $p_{j}\in C^{\infty}_{c}(\R^{d})$ for $j\geqslant -1$, such that 
\begin{enumerate}
\item[1.)]$p_{-1}$ and $p_{0}$ are non-negative radial functions (they just depend on the absolute value of $x\in\R^{d}$), such that the support of $p_{-1}$ is contained in a ball and the support of $p_{0}$ is contained in an annulus;
\item[2.)]$p_{j}(x):=p_{0}(2^{-j}x)$, $x\in\R^{d}$, $j\geqslant 0$;
\item[3.)]$\sum_{j=-1}^{\infty}p_{j}(x)=1$ for every $x\in\R^{d}$; and
\item[4.)]$\operatorname{supp}(p_{i})\cap \operatorname{supp}(p_{j})=\emptyset$ for all $\abs{i-j}>1$.
\end{enumerate}
We then define the Besov spaces 
\begin{align}\label{def:bs}
B^{\theta}_{p,q}:=\{u\in\mathcal{S}':\norm{u}_{B^{\theta}_{p,q}}=\norm[\big]{(2^{j\theta}\norm{\Delta_{j}u}_{L^{p}})_{j\geqslant -1}}_{\ell^{q}}<\infty\},
\end{align}
 where $\Delta_{j}u=\mathcal{F}^{-1}(p_{j}\mathcal{F}u)$ are the Littlewood-Paley blocks, and the Fourier transform is defined with the normalization $\hat{\varphi}(y):=\F\varphi (y):=\int_{\R^{d}}\varphi(x)e^{-2\pi i\langle x,y\rangle}dx$ (and $\F^{-1}\varphi(x)=\hat{\varphi}(-x)$); moreover, $\mathcal{S}$ are the Schwartz functions and $\mathcal S'$ are the Schwartz distributions.
For $p=q=\infty$, the space $B^{\theta}_{\infty,\infty}$ has the unpleasant property that $C^{\infty}_b \subset B^\theta_{\infty,\infty}$ is not dense.
Therefore, we rather work with the following space:
\begin{align*}
\calC^{\theta}:=\{u\in\mathcal{S}'\mid\lim_{j\to\infty}2^{j\theta}\norm{\Delta_{j}u}_{\infty}=0\}
\end{align*} equipped the norm $\norm{\cdot}_{\theta}:=\norm{\cdot}_{B^{\theta}_{\infty,\infty}}$, for which $C^\infty_b$ is a dense subset. We also write $\calC^{\theta}_{\R^{d}}=(\calC^{\theta})^{d}$ and $\calC^{\theta-}:=\cap_{\gamma<\theta}\calC^{\gamma}$.

We recall from Bony's paraproduct theory (cf. \cite[Section 2]{bahouri}) that in general the product $u v:=u\para v+u\arap v +u \reso v$ of $u\in\calC^{\theta}$ and $v\in\calC^{\beta}$ for $\theta,\beta\in\R$, is  well defined if and only if $\theta+\beta>0$. Here, we use the notation of \cite{Martin2017, Mourrat2017Dynamic} for the para- and resonant products $\para, \arap, \reso$, which satisfy the following estimates:
\begin{equation}
\begin{aligned}\label{eq:paraproduct-estimates}
\norm{u\reso v}_{\theta+\beta} & \lesssim\norm{u}_{\theta}\norm{v}_{\beta}, \qquad \text{if }\theta +\beta > 0,\\
\norm{u\para v}_{\beta} \lesssim\norm{u}_{L^{\infty}}\norm{v}_{\beta}& \lesssim\norm{u}_{\theta}\norm{v}_{\beta}, \qquad \text{if } \theta > 0,\\
\norm{u\para v}_{\beta+\theta}& \lesssim\norm{u}_{\theta}\norm{v}_{\beta}, \qquad \text{if } \theta < 0.
\end{aligned}
\end{equation}
So if $\theta + \beta > 0$ we have $\norm{u v}_{\gamma}\lesssim\norm{u}_{\theta}\norm{v}_{\beta}$ for $\gamma:=\min(\theta,\beta,\theta+\beta)$.

For $T>0$, $\rho \in (0,1)$ and for a Banach space $X$ we write $C^{\rho}_{T}X:=C^{\rho}([0,T],X)$, with
\begin{align*}
\norm{u}_{C^{\rho}_{T}X}:=\sup_{0\leqslant s<t\leqslant T}\frac{\norm{u(t)-u(s)}_{X}}{(t-s)^{\rho}}+\sup_{t\in[0,T]}\norm{u(t)}_{X}
\end{align*} and $C_{T}X:=C([0,T],X)$ with norm $\norm{u}_{C_{T}X}:=\sup_{t\in [0,T]}\norm{u(t)}_{X}$.
Analogously, we define for $\overline{T}\in (0,T]$ the space $C_{\overline{T},T}X:=C([T-\overline{T},T],X)$.

Next, we collect some facts about $\alpha$-stable Lévy processes and their generators and semigroups. A symmetric $\alpha$-stable Lévy process $L$ is a Lévy process, satisfying the scaling property $(L_{k t})_{t \geqslant 0}\stackrel{d}{=}k^{1/\alpha}(L_{t})_{t \geqslant 0}$ for any $k>0$ and $L\stackrel{d}{=}-L$, where $\stackrel{d}{=}$ denotes equality in law. These properties determine the jump measure $\mu$ of $L$, see \cite[Theorem 14.3]{sato}.  That is, the Lévy jump measure $\mu$ of $L$ is given by
\begin{align}\label{eq:mu}
\mu(A):=\E\bigg[\sum_{0\leqslant t\leqslant 1}\mathbf{1}_{A}(\Delta L_{t})\bigg]=\int_{S}\int_{\R^{+}}\mathbf{1}_{A}(k\xi)\frac{1}{k^{1+\alpha}}dk\tilde \nu(d\xi),\qquad A\in \mathcal{B}(\R^{d}\setminus\{0\}),
\end{align} where $\tilde \nu$ is a finite, symmetric, non-zero measure on the unit sphere $S\subset\R^{d}$. We also define for $A\in\mathcal{B}(\R^{d}\setminus\{0\})$ and $t\geqslant 0$
\begin{align*}
\pi(A\times [0,t])=\sum_{0\leqslant s\leqslant t}\mathbf{1}_{A}(\Delta L_{s}),
\end{align*} which is a Poisson random measure with intensity measure $dt\mu(dy)$. Let $\hat{\pi}(dr,dy)=\pi(dr,dy)-dr\mu(dy)$ be the compensated Poisson random measure of $L$.
We refer to the book by Peszat and Zabczyk for the integration theory against Poisson random measures and for the Burkholder-Davis-Gundy inequality \cite[Lemma 8.21 and 8.22]{peszat_zabczyk_2007}, which we will both use in the sequel.
The generator $A$ of $L$ satisfies $C_{b}^{\infty}(\R^{d})\subset \operatorname{dom}(A)$ and it is given by
\begin{align}\label{eq:functional}
A\varphi(x)=\int_{\R^{d}}\paren[\big]{\varphi(x+y)-\varphi(x)-\mathbf{1}_{\{\abs{y}\leqslant 1\}}(y) \nabla \varphi(x) \cdot y}\mu(dy)
\end{align} for $\varphi\in C_{b}^{\infty}(\R^{d})$. If $(P_t)_{t\geqslant 0}$ denotes the semigroup of $L$, the convergence $t^{-1}(P_{t}f(x)-f(x))\to Af(x)$ is uniform in $x\in\R^{d}$ (see \cite[Theorem 5.4]{peszat_zabczyk_2007}).

To derive Schauder estimates for $(P_t)$ it will be easier to work with another representation of the generator $A$. For that purpose we first introduce an operator $\La$ via Fourier analysis, and then we show that it agrees with $A$.

\begin{definition}\label{def:fl}
Let $\alpha \in (0,2]$ and let $\nu$ be a symmetric (i.e. $\nu(A)=\nu(-A)$), finite and non-zero measure on the unit sphere $S\subset\R^{d}$. We define the operator $\La$ as
\begin{align}
\La\F^{-1}\varphi=\F^{-1}(\psi^{\alpha}_{\nu} \varphi)\qquad\text{for $\varphi\in C^\infty_b$,}
\end{align}  where
$\psi^{\alpha}_{\nu} (z):=\int_{S}\abs{\langle z,\xi\rangle}^{\alpha}\nu(d\xi).$
\end{definition}

\begin{remark}
If we take $\nu$ as a suitable multiple of the Lebesgue measure on the sphere, then $\psi^\alpha_\nu(z) = |2\pi z|^\alpha$ and thus $\La$ is the fractional Laplace operator $-(-\Delta)^{\alpha/2}$. And if moreover $\alpha=2$, then the fractional Laplacian of course agrees with the usual Laplacian. 
\end{remark}

\begin{lemma}\label{rem:lgen}
For $\varphi \in C^\infty_b$ we have $-\La \varphi = A\varphi$, where $A$ is the generator of the symmetric, $\alpha$-stable Lévy process $L$ with characteristic exponent $\E[\exp(2\pi i\langle z,L_{t}\rangle )]=\exp(-t\psi^{\alpha}_{\nu}(z))$. The process $L$ has the jump measure $\mu$ as defined in~\cref{eq:mu}, with $\tilde \nu = C \nu$ for some $C>0$.
\end{lemma}

\begin{proof}
	By Fourier inversion, $L_{t}$ has the density $\rho_{t}=\F^{-1}(\exp(-t\psi^{\alpha}_{\nu}))$ w.r.t. the Lebesgue measure (note that $\psi^{\alpha}_{\nu}(z)=\psi^{\alpha}_{\nu}(-z)$). So for the semigroup $(P_{t})$ of $L$ we have $P_{t}\varphi(x)=\int\rho_{t}(y)\varphi(x+y)dy$ with $\partial_{t}P_{t}\varphi_{\vert t=0}=\F^{-1}(-\psi^{\alpha}_{\nu}\hat{\varphi})=-\La \varphi$ for any $\varphi\in C^\infty_b$. The identity $\tilde \nu = C \nu$ is shown in the proof of~\cite[Theorem 14.10]{sato}.
\end{proof}

\begin{assumption}\label{ass}
Throughout the paper we assume that the measure $\nu$ from \cref{def:fl} has $d$-dimensional support, in the sense that the linear span of its support is $\R^d$. This means that the process $L$ can reach every open set in $\R^d$ with positive probability.
\end{assumption}

So far we defined $\La$ on $C^\infty_b$, so in particular on Schwartz functions. But the definition of $\La$ on Schwartz distributions by duality is problematic, because for $\alpha \in (0,2)$ the function $\psi^{\alpha}_{\nu}$ has a singularity in $0$. This motivates the next proposition.

\begin{proposition}\textbf{(Continuity of the operator $\La$)}\label{prop:contfl}\\
Let $\alpha\in (0,2]$. Then for $\beta\in\R$ and $u \in C^\infty_b$ we have
\begin{align*}
\norm{\La u}_{\beta-\alpha}\lesssim\norm{u}_{\beta}.
\end{align*}
In particular, $\La$ can be uniquely extended to a continuous operator from $\calC^{\beta}$ to $\calC^{\beta-\alpha}$.
\end{proposition}

\begin{proof}
For $j \geqslant 0$ it follows from~\cite[Lemma~2.2]{bahouri},  that $\| \La \Delta_j u\|_{L^\infty} \lesssim 2^{-j(\beta-\alpha)} \|u \|_\beta$, as $\psi^{\alpha}_{\nu}$ is infinitly often continously differentiable in $\R^{d}\setminus\{0\}$ with $\abs{\partial^{\mu}\psi^{\alpha}_{\nu}(z)}\lesssim\abs{z}^{\alpha-\abs{\mu}}$ for a multiindex $\mu\in\N_{0}^{d}$ with $\abs{\mu}\leqslant\alpha$ and $\Delta_{j}u$ has a Fourier transform, which is supported in $2^{j}\cal{A}$, where $\cal{A}$ is the annulus, where $\rho_{0}$ is supported. For $j=-1$ we use that $-\La = A$ for $A$ as in~\cref{eq:functional}, and therefore
\begin{align*}
	-\La \Delta_{-1} u(x) & =  \int_{\R^d}\left(\Delta_{-1} u(x+y) - \Delta_{-1} u(x) - \nabla \Delta_{-1} u(x)\cdot y \1_{\{|y| \leqslant 1\}} \right) \mu(dy) \\
	& \lesssim \int_{B(0,1)} \|D^2\Delta_{-1} u\|_{L^\infty}|y|^2 \mu(dy) + \|\Delta_{-1} u\|_{L^\infty} \mu(B(0,1)^c) \lesssim \| u \|_{\alpha},
\end{align*}
where $B(0,1) = \{|y| \leqslant 1\}$ and the last step follows from the Bernstein inequality in~\cite[Lemma~2.1]{bahouri}.
\end{proof}
\begin{remark}
One can show that the operators $A$ and $-\La$ even agree on $\bigcup_{\varepsilon>0}\calC^{2+\epsilon}$. Indeed, for $\varphi \in \bigcup_{\varepsilon>0}\calC^{2+\epsilon}$ we have that $\varphi$ and its partial derivatives up to order $2$ are uniformly continuous, and thus it follows from \cite[Theorem~5.4]{peszat_zabczyk_2007} that $A\varphi$ has the same expression as in~\eqref{eq:functional}. Then we can use that $C^{\infty}_{b}$ is dense in $\calC^{2+\epsilon}$ for all $\varepsilon>0$ and apply a continuity argument to deduce that $A\varphi = -\La \varphi$ for $\varphi \in \bigcup_{\varepsilon>0}\calC^{2+\epsilon}$.
\end{remark}

For $z \in \R^d \setminus\{0\}$ we also have
\[
	\psi^{\alpha}_{\nu}(z)= |z|^\alpha \int_{S} \big|\big\langle \frac{z}{|z|},\xi\big\rangle\big|^\alpha \nu(d\xi) \geqslant |z|^\alpha \min_{|y|=1} \int_{S} |\langle y,\xi\rangle|^\alpha \nu(d\xi),
\]
and by~\cref{ass} the minimum on the right hand side is strictly positive. Otherwise, there would be some $y_0\neq 0$ with $\int_S |\langle y_0,\xi\rangle|^\alpha \nu(d\xi) = 0$ and this would mean that the support of $\nu$ (and thus also its span) is contained in the orthogonal complement of $\operatorname{span}(y_0)$. Therefore, $e^{-\psi^\alpha_\nu}$ decays faster than any polynomial at infinity and outside of $0$ it even behaves like a Schwartz function.

\begin{lemma}\label{schauder}
Let $\nu$ be a finite, symmetric measure on the sphere $S\subset\R^{d}$ satisfying \cref{ass}. Let $P_{t}\varphi:=\mathcal{F}^{-1}(e^{-t\psi^{\alpha}_{\nu}}\hat{\varphi}) = \rho_t \ast \varphi$, where $t > 0$, $\rho_t = \mathcal F^{-1} e^{-t\psi^\alpha_\nu} \in L^1$, and $\varphi\in C^\infty_b$. Then we have for $\vartheta\geqslant 0$, $\beta\in\R$
\begin{align}\label{eq:schauder1}
\norm{P_{t}\varphi}_{\beta+\vartheta}\lesssim (t^{-\vartheta/\alpha} \vee 1) \norm{\varphi}_{\beta},
\end{align} and for $\vartheta\in [0,\alpha)$
\begin{align}
\norm{(P_{t}-\operatorname{Id})\varphi}_{\beta-\vartheta}\lesssim t^{\vartheta/\alpha}\norm{\varphi}_{\beta}.\label{eq:schauder2}
\end{align}
Therefore, if $\vartheta\geqslant 0$, then $P_{t}$ has a unique extension to a bounded linear operator in $L(\mathcal{C}^{\beta},\mathcal{C}^{\beta+\vartheta})$ and this extension satisfies the same bounds. 
\end{lemma}

\begin{proof}
	This follows from~\cite[Lemma~A.5]{Gubinelli2015Paracontrolled}, see also \cite[Lemma~A.7, Lemma~A.8]{Gubinelli2015Paracontrolled}.
\end{proof}

\begin{corollary}\textbf{(Schauder Estimates)}\label{cor:schauder}\\
Let $(P_{t})$ and $\nu$ be as in \cref{schauder}. Let $T>0$, $\overline{T}\in (0,T]$, and $\beta\in \R$. For $v\in C_{\overline{T},T}\calC^{\beta}$ and $t\in[T-\overline{T},T]$ we define $J^{T}v (t):=\int_{t}^{T}P_{r-t}v(r)dr$. Then we have for $\vartheta\in [0,\alpha]$
\begin{align}\label{j1}
\norm{J^{T}v}_{C_{\overline{T},T}\calC^{\beta+\vartheta}}\lesssim \overline{T}^{1-\vartheta/\alpha} \norm{v}_{C_{\overline{T},T}\calC^{\beta}}.
\end{align}
If moreover $\beta<0$ and $\vartheta \in(-\beta,\alpha)$, then
\begin{align}\label{j2}
\norm{J^{T}v}_{C^{(\beta +\vartheta)/\alpha}_{\overline{T},T}L^{\infty}}\lesssim \overline{T}^{1-\vartheta/\alpha} \norm{v}_{C_{\overline{T},T}\calC^{\beta}}.
\end{align} 
\end{corollary}

\begin{proof}
	This follows from the same arguments as~\cite[Lemma A.9]{Gubinelli2015Paracontrolled}. In that lemma only the most difficult case $\vartheta = \alpha$ is treated, but the case $\vartheta < \alpha$ follows directly from \cref{schauder} since then $\int_t^T (r-t)^{-\vartheta/\alpha}dr \simeq (T-t)^{1-\vartheta/\alpha} \leqslant \overline T^{1-\vartheta/\alpha}$.
\end{proof}

\end{section}

\begin{section}{The Kolmogorov backward equation}\label{sect:pde}

Our goal is to define and construct weak solutions (or better: martingale solutions) to the Lévy SDE
\begin{align} \label{eq:lsde}
dX_{t}=V(t,X_{t})dt + dL_{t},\quad X_{0}=x\in\R^{d},
\end{align}
 where $L$ is a $d$-dimensional, symmetric, $\alpha$-stable Lévy process for $\alpha\in (1,2]$, and where $V\in C_{T}\mathcal{C}^{\beta}_{\R^{d}}:=C([0,T],(\calC^{\beta})^{d})$, for $\beta<0$. For that purpose we follow~\cite{Delarue2016, Cannizzaro2018} in formulating the martingale problem for $X$, which is based on the operator
\begin{align}
\mathcal{G}^{V}:=\partial_{t} + \La+V\cdot\nabla,
\end{align} where the $\La$ is the generator of $L$ (see \cref{def:fl} and \cref{rem:lgen}). To make sense of the martingale problem, we have to solve the Kolmogorov backward equation
\begin{align}\label{eq:kolmogorov-backward}
\mathcal{G}^{V}u=f,\quad u(T,\cdot)=u^{T}\quad\Leftrightarrow\quad \partial_{t}u=-\La u-V\cdot\nabla u+f,\quad u(T,\cdot)=u^{T},
\end{align}
for $f\in C_{T}\calC^{\epsilon}$ for $\epsilon>0$, and $u^{T}\in\calC^{2+\beta}$. Here we need sufficient regularity of $V$: Since at best $V\cdot\nabla u \in \calC^\beta$ and since inverting $\partial_t + \La$ gains $\alpha$ degrees of regularity, we expect that $u\in C_T  \calC^{\alpha+\beta}$. Thus, we need $\beta + (\alpha+\beta-1)>0$ in order for $V \cdot \nabla u$ to be well defined, i.e. $\beta>\frac{1-\alpha}{2}$ (we call this the \emph{Young case}, by analogy to the Young integral). To allow for more irregular $V$ we follow~\cite{Cannizzaro2018} in using paracontrolled distributions~\cite{Gubinelli2015Paracontrolled}. Then we need to postulate the existence of certain resonant products of $V$, and under that assumption we obtain the existence and uniqueness of a paracontrolled solution $u$ for $\beta>\frac{2-2\alpha}{3}$.

The solution theory for~\cref{eq:kolmogorov-backward} is similar to the Brownian case, where $-\La$ is replaced by the Laplacian, and which is treated in~\cite{Cannizzaro2018}. For completeness we include the proofs, but readers familiar with~\cite{Cannizzaro2018} could skip most of this section and only have a look at~\cref{thm:ygeneq},~\cref{def:roughdist},~\cref{prop:Banachmap} and ~\cref{thm:rgeneq}, where \cref{thm:rgeneq} carries out the arguments for proving the continuity of the solution map in the rough case.

Let us start with the Young case. We call $u$  a mild solution to~\cref{eq:kolmogorov-backward} if
\begin{align*}
u_{t}=P_{T-t}u^{T}+\int_{t}^{T}P_{s-t}(V_{s}\cdot\nabla u_{s}-f_{s})ds=:P_{T-t}u^{T}+J^{T}(V\cdot\nabla u-f)(t),
\end{align*} for $t\in[0,T]$, where $(P_{t})$ is the semigroup generated by $-\La$, as defined in \cref{schauder}.
\begin{theorem}\label{thm:ygeneq}
Let $\alpha\in (1,2]$, $\beta\in(\frac{1-\alpha}{2}, 0)$ and $\theta\in (1-\beta,\beta+\alpha)$. Let $V\in\CTcalC^{\beta}_{\R^{d}}$, $f\in\CTcalC^{\beta}$ and $u^{T}\in \calC^{\theta}$. Then the PDE 
\begin{align}\label{eq:pdem}
\partial_{t}u=\La u-V\cdot \nabla u+f,\quad u(T,\cdot)=u^{T},
\end{align}
admits a unique mild solution $u\in\CTcalC^{\theta}\cap C_{T}^{\theta/\alpha}L^{\infty}$. Moreover, the solution map 
\[
	\calC^{\theta} \times \CTcalC^{\beta} \times \CTcalC^{\beta}_{\R^{d}} \ni (u^{T},f,V)\mapsto u \in \CTcalC^{\theta} \cap C^{\theta/\alpha}_{T}L^{\infty}.
\]
is continuous.
\end{theorem}

\begin{proof}
The proof follows from the Banach fixed point theorem applied to the map
\begin{align*}
\Phi^{\overline T, T} u (t)=P_{T-t}u^{T}+J^{T}(\nabla u\cdot V-f)(t),
\end{align*} where $J^{T}(v)(t)=\int_{t}^{T}P_{r-t}v(r)dr$. Replacing first the interval $[0,T]$ by $[T-\overline{T},T]$ for $\overline{T}\in (0,T]$ suffitiently small, the estimates for $P$ and $J$ from~\cref{schauder} and~\cref{cor:schauder} together with the estimates for the product in~\eqref{eq:paraproduct-estimates} show that if $\overline T$ is sufficiently small, then $\Phi^{\overline T, T}$ is a contraction on $C_{\overline T,T} \calC^{\theta} \cap C^{\theta/\alpha}_{\overline T,T}L^{\infty}$. Moreover, $\overline T$ does not depend on the terminal condition $u^T$ and therefore we can iterate this construction and patch the solutions together to obtain a solution on $[0,T]$.

The continuity of the solution map follows from the linearity of the equation and from Gronwall's inequality for locally finite measures, cf. \cite[Appendix, Theorem 5.1]{ethier}.
\end{proof}

Our next aim is to go beyond the Young case. If $\beta\leqslant \frac{1-\alpha}{2}$, then the sum  of the regularities of $\nabla u$ and $V$ is negative ($\theta - 1 + \beta < \beta+\alpha-1+\beta\leqslant 0$), and therefore the resonant product $\nabla u \reso V$ is ill defined. To overcome this problem, we use the \textit{paracontrolled ansatz}
\begin{align}
u&=
\nabla u\para J^{T}(V)+u^{\sharp},\label{eq:pansatz}
\end{align}
where the paraproduct is defined as $\nabla u\para J^{T}(V) = \sum_{j=1}^d \partial_j u \para J^T V^j$, and where $u^\sharp$ will be more regular than $u$.

\begin{remark}\label{rem:intuition} The intuition behind the paracontrolled ansatz is as follows. Assume that we found a solution $u\in C_{T}^{\theta/\alpha}L^{\infty}\cap C_{T}\calC^{\theta}$  for $\theta=\beta+\alpha-\epsilon$ for some (small) $\epsilon>0$,
and that we can make sense of the resonant product $\nabla u\reso V$ in such a way that it has its natural regularity $C_{T}\calC^{\beta+\theta-1}$, despite the fact that $\beta+\theta-1\leqslant 0$. Then we would get that
\begin{align*}
u^{\sharp}:&=u-\nabla u\para J^{T}(V)\\&=P_{T-\cdot}u^{T}-J^{T}(f)+J^{T}(\nabla u\arap V)+J^{T}(\nabla u\reso V)+(J^{T}(\nabla u\para V)-\nabla u\para J^{T}(V))
\end{align*} is more regular than $u$ (in fact $2\theta-1$ regular in space, if $u^{T}\in C_{T}\calC^{2\theta-1}$ and $f\in C_{T}\calC^{\epsilon}$ for $\epsilon>0$) by Schauder estimates for the first four terms and by the commutator estimate from \cref{lem:sharp} below for last term on the right hand side. This explains why  the paracontrolled ansatz might be \emph{justified}. The reason why the ansatz is \emph{useful} is that it isolates the singular part of $u$ in a paraproduct, and then we can use commutator estimates to handle the paraproduct.
\end{remark}

Therefore, we have to show that assuming the paracontrolled ansatz we can make sense of the product $\nabla u\reso V$ (by moreover postulating the existence of certain extrinsically given resonant products of $V$) and that the paracontrolled ansatz is stable under the Banach fixed point map. To make this precise, we need to define the Banach space of paracontrolled distributions. From now on we fix $\alpha \in (1,2]$ and $\beta \in (\tfrac{2-2\alpha}{3}, 0)$ and we define paracontrolled distributions as follows:

\begin{definition}\label{def:paradist}
Let $T>0$ let $ V\in C_{T}\mathcal{C}^{\beta}_{\R^{d}}$ be fixed. For $\theta \in ((2-\beta)/2,\beta+\alpha)$ and $\overline{T}\in (0,T]$, we define the space of \emph{paracontrolled distributions} $\mathcal{D}_{\overline{T},T}^{\theta} = \mathcal{D}_{\overline{T},T}^\theta(V)$ as the set of tuples $(u,u')\in (C_{\overline{T},T}\calC^{\theta}\cap C_{\overline{T},T}^{\theta/\alpha}L^{\infty})\times C_{\overline{T},T}\calC^{\theta-1}_{\R^{d}}$ such that $u^{\sharp}:=u-u'\para J^{T}(V)\in C_{\overline{T},T}\calC^{2\theta-1}$. We define a norm on $\mathcal{D}_{\overline{T},T}^{\theta}$ by setting
\begin{align*}
\norm{(u,u')-(v,v')}_{\mathcal{D}_{\overline{T},T}^{\theta}} :=\norm{u-v}_{\calD} &:= \norm{u-v}_{C_{\overline{T},T}\mathcal{C}^{\theta}}+\norm{ u- v}_{C^{\theta/\alpha}_{\overline{T},T}L^{\infty}}
\\
	&\qquad
		+\norm{u'-v'}_{C_{\overline{T},T}\mathcal{C}^{\theta-1}_{\R^{d}}}+\norm{u^{\sharp}-v^{\sharp}}_{C_{\overline{T},T}\mathcal{C}^{2\theta-1}}.
\end{align*} 
Then, $(\mathcal{D}_{\overline{T},T}^{\theta},\norm{\cdot}_{\mathcal{D}^{\theta}_{\overline{T},T})}$ is a Banach space. If moreover $W \in C_{T}\mathcal{C}^{\beta}_{\R^{d}}$ and $(v,v') \in \mathcal{D}_{\overline{T},T}^\theta(W)$, then we use the same notation $\norm{(u,u')-(v,v')}_{\mathcal{D}_{\overline{T},T}^{\theta}}$ or $\norm{u-v}_{\calD}$ with the same definition, despite the fact that $(u,u')$ and $(v,v')$ do not live in the same space.
\end{definition}

\begin{remark}
In contrast to the definition of Cannizzaro and Chouk, we included the norm $\norm{u}_{C_{T}^{\theta/\alpha}L^{\infty}}$ instead of $\norm{\nabla u}_{C_{T}^{(\theta-1)/\alpha}L^{\infty}_{\R^{d}}}$, as it will be easier to show continuity of the solution map w.r.t. the $C_{T}^{\theta/\alpha}L^{\infty}$-norm, which will be needed below. 
Moreover, our space of paracontrolled distributions does not depend on the right hand side $f$.
\end{remark}

If we assume that $u$ is paracontrolled, then we can make sense of the problematic term $\nabla u\reso V$, despite the fact that $u$ has insufficient regularity: We have
\begin{align}\label{eq:derivation}
\partial_{j} u\reso V^{j}&=\sum_{i=1}^{d}(u^{\prime,i}\para J^{T}(\partial_{j}V^{i}))\reso V^{j}+U^{\sharp ,j}\reso V^{j}\nonumber\\&=\squeeze[2]{\sum_{i=1}^{d}u^{\prime,i}(J^{T}(\partial_{j}V^{i})\reso V^{j})+\sum_{i=1}^{d}\mathcal{R}(u^{\prime,i},J^{T}(\partial_{j}V^{i}),V^{j})+U^{\sharp ,j}\reso V^{j}},
\end{align} where we define
\begin{align}\label{eq:r}
\mathcal{R}(f,g,h):=(f\para g)\reso h-f(g\reso h)
\end{align} and $U^{\sharp ,j}= \partial_{j} u^{\sharp}+\sum_{i=1}^{d}\partial_{j}u^{\prime,i}\para J^{T}(V^{i})\in C_{T}\calC^{2\theta-2}$. By the commutator lemma \cite[Lemma 2.4]{Gubinelli2015Paracontrolled} the term $\mathcal{R}(u^{\prime,i},J^{T}(\partial_{j}V^{i}),V^{j})$ is well defined and in $C_T\calC^{2\theta-2+\beta}$. The term $J^{T}(\partial_{j}V^{i}))\reso V^{j}$ is still ill defined, but it only depends on $V$. So let us assume that we are extrinsically given for all $i$ and $j$ the resonant products $J^{T}(\partial_{j}V^{i})\reso V^{j}\in \CTcalC^{2\beta+\alpha-1}$. Then the product $u^{\prime,i}(J^{T}(\partial_{j}V^{i})\reso V^{j})$ is well defined since $\theta-1 + 2\beta+\alpha-1>0$ and thus the product $\nabla u\reso V=\sum_{j=1}^{d}\partial_{j} u\reso V^{j}$ is well defined in $C_{T}\calC^{2\theta-2+\beta}$.

This discussion motivates the following definition:

\begin{definition}\textbf{(Enhanced drift)}\label{def:roughdist}\\
Let $\beta\in (\frac{2-2\alpha}{3},0)$ and $T>0$. For $\beta \in (\frac{1-\alpha}{2},0)$ we define the space $\mathcal X^\beta := C_T \calC^\beta_{\R^d}$. For $\beta\in (\frac{2-2\alpha}{3},\frac{1-\alpha}{2}]$ we define $\mathcal{X}^{\beta}$ as the closure of
\begin{align*}
\{\calK (\eta ):= (\eta,(J^{T}(\partial_{j}\eta^{i})\reso\eta^{j})_{i,j=1,...,d}):\, \eta\in C_T C^{\infty}_{b}(\R^{d},\R^{d})\}
\end{align*}
in $\CTcalC^{\beta}_{\R^{d}}\times \CTcalC^{2\beta+\alpha-1}_{\R^{d^{2}}}$
In that case we will also denote the elements of $\mathcal{X}^{\beta}$ by $\mathcal{V}=(\mathcal{V}_{1},\mathcal{V}_{2})$, and we say that $\mathcal{V}$ is a \emph{lift} or an \emph{enhancement} of $V$ if $\mathcal{V}_{1}=V$.
\end{definition}

\begin{proposition}\label{prop:product}
Let $T>0$, $\overline{T}\in (0,T)$ and $(2-\beta)/2<\theta<\beta+\alpha$ and $\beta\in (\frac{2-2\alpha}{3},\frac{1-\alpha}{2}]$. For $(u,u')\in\mathcal{D}_{\overline{T},T}^{\theta}$ and $\mathcal{V}=(\mathcal{V}_{1},\mathcal{V}_{2})\in \mathcal{X}^{\beta}$, we define 
\begin{align*}
\nabla u\reso \mathcal{V}:=\sum_{i,j=1}^{d}u^{\prime,i}\mathcal{V}_{2}^{i,j}+\sum_{i,j=1}^{d}\mathcal{R}(u^{\prime,i},J^{T}(\partial_{j}\mathcal{V}_{1}^{i}),\mathcal{V}_{1}^{j})+\sum_{j=1}^{d}U^{\sharp,j}\reso\mathcal{V}_{1}^{j}.
\end{align*}
Here, the commutator $\mathcal{R}$ is as in \eqref{eq:r} and $U^{\sharp,j}=\partial_{j}u^{\sharp}+\sum_{i=1}^{d}\partial_{j}u^{\prime,i}\para J^{T}(\mathcal{V}_{1}^{i})\in C_{\overline T, T} \calC^{2\theta-2}$. Then, the map $\mathcal{D}_{\overline{T},T}^{\theta}\ni(u,u')\mapsto\nabla u\reso V\in C_{\overline{T},T}\mathcal{C}^{2\beta+\alpha-1}$ is Lipschitz continuous, more precisely
\begin{align}\label{eq:2}
\MoveEqLeft
\norm{\nabla u\reso\mathcal{V}-\nabla v\reso\mathcal{V}}_{C_{\overline{T},T}\mathcal{C}^{2\beta+\alpha-1}}\lesssim \norm{\mathcal{V}}_{\mathcal{X}^{\beta}}(1+\norm{\mathcal{V}}_{\mathcal{X}^{\beta}})\norm{(u,u')-(v,v')}_{\mathcal{D}_{\overline{T},T}^{\theta}}.
\end{align}
Moreover, the product $\nabla u\cdot\mathcal{V}:=\nabla u\para \mathcal{V}_{1}+\nabla u\arap\mathcal{V}_{1}+\nabla u\reso\mathcal{V}$, where $\nabla u\reso\mathcal{V}$ is defined as above, is well defined in $\CTcalC^{\beta}$.
\end{proposition}

\begin{proof}
	The products $u^{\prime,i}\mathcal{V}_{2}^{i,j}$ are well defined because the sum of the regularities is $\theta-1+2\beta+\alpha-1 > -1 + \tfrac32 \beta + \alpha > 0$. The commutators $\mathcal{R}(u^{\prime,i},J^{T}(\partial_{j}\mathcal{V}_{1}^{i}),\mathcal{V}_{1}^{j})$ are well defined because the sum of the regularities is $\theta-1 + \beta + \alpha -1 + \beta  > -1 + \tfrac32 \beta + \alpha > 0$. The resonant products $U^{\sharp,j}\reso\mathcal{V}_{1}^{j}$ are well defined since the sum of the regularities is $2\theta-2 + \beta > 2 - \beta -2 + \beta = 0$.
\end{proof}
We already motivated in \cref{rem:intuition}, that we will need the following commutator lemma concerning the action of the $J^{T}$ operator on the $\para$-paraproduct. Its proof can be found in \cref{sec:commutator-appendix}.

\begin{lemma}\label{lem:sharp}
Let $T>0$, $0<\sigma<1$, $\varsigma\in\R$ with $-1\leqslant\sigma-\varsigma+1<\alpha$ and $h\in \CTcalC^{\varsigma}_{\R^{d}}$. For $\overline{T}\in (0,T]$, let $g\in C^{\sigma /\alpha}_{\overline{T},T}L^{\infty}_{\R^{d}}\cap C_{\overline{T},T}\mathcal{C}^{\sigma}_{\R^{d}}$. Then the following inequality holds
\begin{align*}
\norm{J^{T}( g\para h)- g\para J^{T}(h)}_{C_{\overline{T},T}\mathcal{C}^{2\sigma+1}}\lesssim \overline{T}^{\,\kappa}(\norm{g}_{C_{\overline{T},T}\mathcal{C}^{\sigma}_{\R^{d}}}+\norm{g}_{C^{\sigma /\alpha}_{\overline{T},T}L^{\infty}_{\R^{d}}})\norm{h}_{C_{T}\mathcal{C}^{\varsigma}_{\R^{d}}}
\end{align*} where $\kappa=1-\frac{\sigma+1-\varsigma}{\alpha}>0$.
\end{lemma}

For fixed $\mathcal{V}\in\mathcal{X}^{\beta}$ and $f\in\CTcalC^{\epsilon}$, $\epsilon>0$, or $f=\mathcal{V}_{1}^{j}$ for some $j$, the contraction mapping will now be defined as
\begin{align*}
\Phi^{\overline{T},T}:\,&\mathcal{D}_{\overline{T},T}^{\theta}\longrightarrow\mathcal{D}_{\overline{T},T}^{\theta},\qquad(u,u')\mapsto (v,v'),
\end{align*} where
\begin{align}
v:=-J^{T}(f)+J^{T}(\nabla u\cdot\mathcal{V})+\psi^{T}
\end{align} for $\psi^{T}_{t}=P_{T-t}u^{T}$ and
\begin{align*}
v':=\begin{cases}
\nabla u\quad &\text{ if } f\in C_{T}\calC^{\epsilon}\\
\nabla u-e_{j}\quad &\text{ if } f=\mathcal{V}_{1}^{j}
\end{cases},
\end{align*}
where $(e_{i})$ is the canonical basis of $\R^{d}$.

\begin{theorem}\label{prop:Banachmap}
Let $T>0$, $\beta\in (\frac{2-2\alpha}{3},\frac{1-\alpha}{2}]$ and $\theta \in ((2-\beta)/2,\beta+\alpha)$. Let $u^{T}\in\mathcal{C}^{2\theta-1}$, $f\in\CTcalC^{\epsilon}$ or $f=\mathcal{V}_{1}^{j}$ for some $j$ and let $\mathcal{V}\in\mathcal{X}^{\beta}$. Then there exists a unique fixed point of the map $\Phi^{T,T}$ in $\mathcal{D}^{\theta}_{T,T}$, that is, a unique (mild) solution of the Kolmogorov backward equation 
\begin{align}\label{eq:rpde}
\mathcal{G}^{(\mathcal{V}_{1},\mathcal{V}_{2})}u=f,\qquad u(T,\cdot)=u^{T},
\end{align} where $\mathcal{G}^{(\mathcal{V}_{1},\mathcal{V}_{2})}:=\partial_{t}+\La+\mathcal{V}\cdot \nabla $. Moreover, for $V \in C_T (C^{\infty}_b)^{d}$ and $f\in C_T C_{b}$ the solution $u$ of $\mathcal{G}^{(V,\calK(V))}u=f$, $u(T,\cdot)=u^{T}$, agrees with the classical solution of the PDE. 
\end{theorem}

\begin{proof}
We first consider $f\in C_{T}\calC^{\epsilon}$ for $\epsilon>0$, and we show that for $\overline{T} \in (0,T]$ and for $(u,u')\in\calD$ we have $\Phi^{\overline{T},T}(u,u')\in\calD$, and that there exists $\kappa>0$, depending only on $\theta$ and $\beta$, such that
\begin{align}\label{eq:contraction}
\norm{\Phi^{\overline{T},T}(u,u')-\Phi^{\overline{T},T}(\tilde u,\tilde u')}_{\calD}\lesssim (1+\norm{\mathcal{V}}_{\mathcal{X}^{\beta}})\norm{\mathcal{V}}_{\mathcal{X}^{\beta}}\,\overline{T}^{\,\kappa}\norm{(u,u')-(\tilde u,\tilde u')}_{\calD},
\end{align} 
so in particular that $\Phi^{\overline T,T}$ is a strict contraction for sufficiently small $\overline{T}$. By linearity of $\Phi^{\overline T,T}$ it suffices to estimate $\norm{\Phi^{\overline{T},T}(u,u')}_{\calD}$.

So let $\Phi^{\overline{T},T}(u,u')=(v,v')$. We need to bound the norms 
\begin{align*}
\norm{v}_{C_{\overline{T},T}\calC^{\theta}},\quad\norm{v}_{C_{\overline{T},T}^{\theta/\alpha}L^{\infty}},\quad\norm{v'}_{C_{\overline{T},T}\calC^{\theta-1}_{\R^{d}}},\quad\norm{\Phi^{\overline{T},T}(u,u')^{\sharp}}_{C_{\overline{T},T}\calC^{2\theta-1}},
\end{align*} where $\Phi^{\overline{T},T}(u,u')^{\sharp}=v-v'\para J^{T}(\mathcal{V}_{1})=v-\nabla u\para J^{T}(\mathcal{V}_{1})$. We only show the estimate for $\norm{v}_{C_{\overline{T},T}^{\theta/\alpha}L^{\infty}}$. The other terms can be estimated using the same arguments as in \cite[Proposition 3.9]{Cannizzaro2018}, the only difference is that we use the Schauder estimates for $-\La$ instead of those for the Laplacian. 
For $T-\overline{T}\leqslant r<t\leqslant T$ we have by \eqref{eq:schauder2} and \eqref{j2}
\begin{align}\label{eq:b}
\MoveEqLeft
\norm{v(t)-v(r)}_{L^{\infty}}\nonumber\\&\leqslant\squeeze[2]{\norm{\psi^{T}_{t}-\psi^{T}_{r}}_{L^{\infty}}+\norm{J^{T}(\nabla u\cdot\mathcal{V})(t)-J^{T}(\nabla u\cdot \mathcal{V})(r)}_{L^{\infty}}+\norm{J^{T}(f)(t)-J^{T}(f)(r)}_{L^{\infty}}}\nonumber\\&\lesssim\overline{T}^{\kappa_{1}}\abs{t-r}^{\theta/\alpha}\paren[\bigg]{\norm{u^{T}}_{\beta+\alpha}+\norm{\nabla u\cdot\mathcal{V}}_{C_{\overline{T},T}\calC^{\beta}}+\norm{f}_{C_{\overline{T},T}\calC^{\beta}}}\nonumber\\&\lesssim \overline{T}^{\kappa_{1}}\abs{t-r}^{\theta/\alpha}\paren[\bigg]{\norm{u^{T}}_{\beta+\alpha}+(1+\norm{\mathcal{V}}_{\mathcal{X}^{\beta}})\norm{\mathcal{V}}_{\mathcal{X}^{\beta}}\norm{u}_{\calD}+\norm{f}_{C_{\overline{T},T}\calC^{\beta}}}
\end{align} where $\kappa_{1}=\frac{\beta+\alpha-\theta}{\alpha}$ and where we used the estimate for the product $\nabla u\cdot\mathcal{V}=\nabla u\reso\mathcal{V}_{1}+\nabla u\para\mathcal{V}_{1}+\nabla u\arap\mathcal{V}_{1}$ from \cref{prop:product}. Thus,  we have
\begin{align*}
\norm{v}_{C_{\overline{T},T}^{\theta/\alpha}L^{\infty}}\lesssim \overline{T}^{\kappa_{1}}\paren[\bigg]{\norm{u^{T}}_{\beta+\alpha}+\norm{f}_{C_{\overline{T},T}\calC^{\epsilon}}+(1+\norm{\mathcal{V}}_{\mathcal{X}^{\beta}})\norm{\mathcal{V}}_{\mathcal{X}^{\beta}}\norm{u}_{\calD}}.
\end{align*}

To obtain a mild solution $u$ on $[0,T]$, we solve the equation iteratively first on $[T-\overline{T},T]$ with terminal condition $u^{T}\in\calC^{2\theta-1}$, then on $[T-2\overline{T},T-\overline{T}]$, and so on. There is a small subtlety because also on $[T-2\overline{T},T-\overline{T}]$ we will consider solutions that are  paracontrolled by $J^{T}(V)$ and not by $J^{\overline T, T}$. Moreover, the terminal condition $u(T-\overline{T},\cdot)$ is only in $\calC^\theta$ and not in $\calC^{2\theta-1}$. But we only needed $u^T \in \calC^{2\theta-1}$ in order to obtain a regular terminal condition $u^\sharp(T) \in \calC^{2\theta-1}$. And on the interval $[T-2\overline{T},T-\overline{T}]$ we have the terminal condition $u(T-\overline T) - \nabla u(T-\overline T)\para J^T V(T - \overline T)$, which is in $\calC^{2\theta-1}$ since $u$ is paracontrolled on $[T-\overline T, T]$. By iterating this, we obtain a unique fixed point of the map $\Phi^{T,T}$  and thus the unique paracontrolled solution of the equation on $[0,T]$.

The case $f=\mathcal{V}_{1}^{j}$ for some $j \in \{1,\dots, d\}$ is similar and we omit the argument.\\
In the case of $V \in C_T (C^{\infty}_b)^{d}$ and $f\in C_T C_{b}$ the solution $u$ of $\mathcal{G}^{(V,\calK(V))}u=f$, $u(T,\cdot)=u^{T}$ agrees with the classical solution of the PDE, as the product $\calK(V)=V\circ J^{T}(\nabla V)$ is well-defined (in any Besov space with positive regularity) and the product $\mathcal{V}\cdot\nabla u$ agrees with the usual product $V\cdot\nabla u$ by the derivation in \eqref{eq:derivation}.
\end{proof}

\begin{remark}\label{rem:nablau}
If $u\in C_{T}^{\theta/\alpha}L^{\infty}\cap C_{T}\calC^{\theta}$, then $\nabla u\in C_{T}^{(\theta-1)/\alpha}L^{\infty}_{\R^{d}}$. Indeed, we estimate the Littlewood-Paley blocks in two different ways, once using the time regularity of $u$ and then the space regularity to interpolate between the two bounds. That is, we have
\begin{align*}
\norm{\Delta_{j}(u_{t}-u_{s})}_{L^{\infty}}\lesssim \abs{t-s}^{\theta/\alpha}\norm{u}_{C_{T}^{\theta/\alpha}L^{\infty}} \wedge 2^{-j\theta}\norm{u}_{C_{T}\calC^{\theta}}
\end{align*}
and thus for $\abs{t-s}\leqslant 1$
\begin{align*}
\norm{\nabla u_{t}-\nabla u_{s}}_{L^{\infty}_{\R^{d}}}&\lesssim\sum_{j}\norm{\Delta_{j}(\nabla u_{t}-\nabla u_{s})}_{L^{\infty}_{\R^{d}}} \lesssim \sum_{j}2^{j}\norm{\Delta_{j}(u_{t}-u_{s})}_{L^{\infty}}\\&\lesssim\sum_{j:2^{-j}\geqslant\abs{t-s}^{1/\alpha}}\abs{t-s}^{\theta/\alpha}2^{j}+\sum_{j:2^{-j}<\abs{t-s}^{1/\alpha}}2^{-j(\theta -1)}\\&\lesssim\abs{t-s}^{\theta/\alpha-1/\alpha}+\abs{t-s}^{(\theta -1)/\alpha}=2\abs{t-s}^{(\theta-1)/\alpha},
\end{align*} using that $\theta>1$ for the convergence of the geometric series in the estimate of the second summand.
\end{remark}

\begin{theorem}\label{thm:rgeneq}
In the setting of Theorem~\ref{prop:Banachmap}, the solution map 
\[
	(u^{T},\mathcal{V}, f)\in\calC^{2\theta-1} \times \mathcal{X}^{\beta} \times (\CTcalC^{\epsilon} \cup \{\mathcal V_1^1,\dots, \mathcal V_1^j\}) \mapsto u\in \CTcalC^{\theta}\cap C_{T}^{\theta/\alpha}L^{\infty},
\]
where $u$ is the solution of \eqref{eq:rpde} and $\CTcalC^{\theta}\cap C_{T}^{\theta/\alpha}L^{\infty}$ is equipped with the sum of the respective norms, is locally Lipschitz continuous.
\end{theorem}

\begin{proof}
We only show the continuity for $f\in C_{T}\calC^{\epsilon}$, the case $f=\mathcal{V}_{1}^{i}$ is handled analogously. The continuity of the solution map is a bit subtle, because the space $\calD(V)$ depends on $V$. Let $u$ be the solution of the PDE for $\mathcal{V}\in\mathcal{X}^{\beta}$, $f\in\CTcalC^{\epsilon}$ and $u^{T}\in\calC^{2\theta-1}$ and $v$ the solution corresponding to the data $\mathcal{W}\in\mathcal{X}^{\beta}$, $g\in\CTcalC^{\epsilon}$ and $v^{T}\in\calC^{2\theta-1}$. By the fixed point property we have $\Phi^{T,T}(u,u')=(u,u')$ and $\Phi^{T,T}(v,v')=(v,v')$. We want to estimate $\norm{u - v}_{\mathcal D^\theta_{T,T}}$ by itself times a factor less than $1$, plus a term depending on $\norm{f-g}$, $\norm{\mathcal{V}-\mathcal{W}}$ and $\norm{u^{T}-v^{T}}$. Here we keep in mind that $u\in\mathcal D^\theta_{T,T}(V)$, whereas $v\in\mathcal D^\theta_{T,T}(W)$, but we explained the notation of $\norm{u - v}_{\mathcal D^\theta_{T,T}}$ in \cref{def:paradist}. For that purpose we estimate using the definition of the product from \cref{prop:product} and rebracketing like $ab-cd=a(b-d)+(a-c)d$, 
\begin{align*}
\MoveEqLeft
\norm{\nabla u\cdot\mathcal{V}-\nabla v\cdot\mathcal{W}}_{C_{T}\mathcal{C}^{\beta}}
\\&\lesssim (1+\norm{\mathcal{W}}_{\mathcal{X}^{\beta}})\norm{\mathcal{V}}_{\mathcal{X}^{\beta}}\norm{u-v}_{\mathcal{D}^{\theta}_{T,T}}+(1+\norm{\mathcal{W}}_{\mathcal{X}^{\beta}})\norm{\mathcal{V}-\mathcal{W}}_{\mathcal{X}^{\beta}}\norm{v}_{\mathcal{D}^{\theta}_{T,T}}\\&\qquad + \norm{\mathcal{V}}_{\calX^{\beta}}\norm{u}_{\mathcal{D}^{\theta}_{T,T}}\norm{\mathcal{V}-\mathcal{W}}_{\calX^{\beta}}.
\end{align*}
Since the solution $u$ can be bounded in terms of $u^{T},f,\mathcal{V}$ by Gronwall's inequality for locally finite measures (cf. \cite[Appendix, Theorem 5.1]{ethier}), and similarly for $v$, we conclude that
\begin{align*}
\MoveEqLeft
\norm{\nabla u\cdot\mathcal{V}-\nabla v\cdot\mathcal{W}}_{C_{T}\mathcal{C}^{\beta}}
\\&\lesssim\paren[\big]{(\norm{\mathcal V}_{\calX^{\beta}}+\norm{v}_{\mathcal{D}^{\theta}_{T,T}})(1+\norm{\mathcal W}_{\calX^{\beta}})+\norm{\mathcal V}_{\calX^{\beta}}\norm{u}_{\mathcal{D}^{\theta}_{T,T}}}\paren[\bigg]{\norm{u-v}_{\mathcal{D}^{\theta}_{T,T}}+\norm{\mathcal{V}-\mathcal{W}}_{\calX^{\beta}}}\\&\lesssim C(\norm{\mathcal V},\norm{\mathcal W},\norm{u^{T}},\norm{v^{T}},\norm{f},\norm{g})\paren[\bigg]{\norm{u-v}_{\mathcal{D}^{\theta}_{T,T}}+\norm{\mathcal{V}-\mathcal{W}}_{\calX^{\beta}}},
\end{align*} where $C=C(\norm{\mathcal V},\norm{\mathcal W},\norm{u^{T}},\norm{v^{T}},\norm{f},\norm{g})$ is a constant, that depends on the norms of the input data on $[0,T]$.
Therefore, we obtain from~\cref{schauder} together with the Schauder estimates~\cref{cor:schauder} and~\cref{rem:nablau} (recall that $u'=\nabla u$ and $v'=\nabla v$):
\begin{align*}
	&\norm{u-v}_{C_{T}\calC^{\theta}} + \norm{u-v}_{C_{T}^{\theta/\alpha}L^\infty} + \norm{u'-v'}_{C_T \calC^{\theta-1}_{\R^{d}}} \\
	&\hspace{40pt}\lesssim\norm{u^{T}-v^{T}}_{2\theta-1}+\norm{f-g}_{\CTcalC^{\epsilon}}+\norm{J^{T}(\nabla u\cdot\mathcal{V}-\nabla v\cdot\mathcal{W})}_{\CTcalC^{\theta}\cap C_{T}^{\theta/\alpha}L^{\infty}}\\
	&\hspace{40pt}\lesssim\norm{u^{T}-v^{T}}_{2\theta-1}+\norm{f-g}_{\CTcalC^{\epsilon}}+T^{\kappa_{1}}C\paren[\big]{\norm{u-v}_{\mathcal D^\theta_{T,T}}+\norm{\mathcal{V}-\mathcal{W}}_{\calX^{\beta}}},
\end{align*}
where $\kappa_{1}=\frac{\beta+\alpha-\theta}{\alpha}>0$ and where the $C_{T}^{\theta/\alpha}L^{\infty}$-norm of $u-v$ is estimated using the fixed point and an estimate as in \eqref{eq:b}.
Moreover, using the fixed point property and analogue estimates for the term $\norm{\Phi^{T,T}(u,u')^{\sharp}}_{C_{T}\calC^{2\theta-1}}$ as in the proof of \cite[Proposition 3.9]{Cannizzaro2018} using \cref{lem:sharp}, we obtain
\begin{align*}
\norm{u^{\sharp}-v^{\sharp}}_{C_{T}\calC^{2\theta-1}}&\squeeze[1]{\lesssim \norm{u^{T}-v^{T}}_{2\theta-1}+\norm{f-g}_{C_{T}\calC^{\epsilon}}+(1+C)\norm{\mathcal{V}-\mathcal{W}}_{\mathcal{X}^{\beta}}+(T^{\kappa_{2}}\vee T^{\kappa_{1}})C\norm{u-v}_{\mathcal D^\theta_{T,T}}},
\end{align*} where $\kappa_{2}=\frac{2(\alpha+\beta-\theta)}{\alpha}>0$ and $C> 0$ is again a (possibly different) constant depending on the norms of the input data.
So overall
\begin{equation*}
\norm{u-v}_{\mathcal D^\theta_{T,T}}\lesssim \norm{u^{T}-v^{T}}_{2\theta-1}+\norm{f-g}_{C_{T}\calC^{\epsilon}}+(1+C)\norm{\mathcal{V}-\mathcal{W}}_{\mathcal{X}^{\beta}}+(T^{\kappa_1} \vee T^{\kappa_2})C\norm{u-v}_{\mathcal D^\theta_{T,T}}.
\end{equation*}
Assume for the moment that $T$ is small enough so that $T^{\kappa_1} \vee T^{\kappa_2}$ times the implicit constant on the right hand side is $<1$. Then we can take the last term to the other side and divide by a positive factor, obtaining
\begin{align*}
\norm{u-v}_{\mathcal D^\theta_{T,T}}\lesssim\tilde C(\norm{u^{T}-v^{T}}_{2\theta-1}+\norm{f-g}_{C_{T}\calC^{\epsilon}}+\norm{\mathcal{V}-\mathcal{W}}_{\mathcal{X}^{\beta}}),
\end{align*} where $\tilde{C}>0$ is a constant that depends on the input data.
Thus, the map $(u^{T},f,\mathcal V)\mapsto (u,u^{\sharp})\in C_{T}\calC^{\theta}\cap C_{T}\calC^{2\theta-1}$ is locally Lipschitz continuous, which implies that the solution map is continuous with values in $C_{T}\cal C^{\theta}$.

If $T$ is such that $T^{\kappa_1} \vee T^{\kappa_2}$ times the implicit constant is $\geqslant 1$, then we apply the same estimates for $\calD$, where $\overline{T}$ is small enough, and then bound $\norm{u-v}_{C_{T}\calC^{\theta}}\leqslant\sum_{i=1}^{n}\norm{u-v}_{C_{\overline{T},T-(i-1)\overline{T}}\calC^{\theta}}$, where $n$ is the smallest integer such that $T-n\overline{T}\leqslant 0$. The same argument also works for the $C_{T}^{\theta/\alpha}L^{\infty}$-norm with $\norm{u-v}_{C_{T}^{\theta/\alpha}L^{\infty}}\lesssim\sum_{i=1}^{n}\norm{u-v}_{C_{\overline{T},T-(i-1)\overline{T}}^{\theta/\alpha}L^{\infty}}$ for the chosen $n$ and we obtain also local Lipschitz continuity of the solution map w.r.t. this norm.
\end{proof}
\end{section}

\begin{section}{Existence and uniqueness for the martingale problem}\label{sect:mthm}

Recall the definition of $\mathcal X^\beta$ from~\cref{def:roughdist}: For $\beta \in (\frac{1-\alpha}{2}, 0)$ we have $\mathcal X^\beta = C_T \calC^\beta_{\R^d}$, while for $\beta \in (\frac{2-2\alpha}{3},\frac{1-\alpha}{2}]$ the space $\mathcal X^\beta$ is the closure of
\[
	\{\calK (\eta ):= (\eta,(J^{T}(\partial_{j}\eta^{i})\reso\eta^{j})_{i,j=1,...,d}):\, \eta\in C_T C^{\infty}_{b}(\R^{d},\R^{d})\}
\]
in $\CTcalC^{\beta}_{\R^{d}}\times \CTcalC^{2\beta+\alpha-1}_{\R^{d^{2}}}$. For $V \in \mathcal X^\beta$ we define solutions to the SDE
\[
	dX_{t}=V(t,X_{t})dt + dL_{t},\quad X_{0}=x\in\R^{d},
\]
as solutions to the corresponding martingale problem.

We consider the Skorokhod space $(\Omega,\mathcal{F}):=(D([0,T],\R^{d}),\mathcal{B}(D([0,T],\R^{d})))$ with canonical filtration $(\mathcal{F}_{t})_{t\geqslant 0}$, i.e. $\mathcal{F}_{t}=\sigma(X_{s}:s\leqslant t)$ where $(X_{t})_{t\geqslant 0}$ is the canonical process with $X_{t}=\omega(t)$ for $\omega\in\Omega$.

\begin{definition}\label{def:martp}\textbf{(Martingale Problem)}\\
Let $\alpha\in (1,2]$ and $\beta\in(\frac{2-2\alpha}{3},0)$, and let $T>0$ and $V\in \mathcal{X}^{\beta}$. Then, we call a probability measure $\p$ on the Skorokhod space $(\Omega,\mathcal{F})$ a solution of the martingale problem for $(\mathcal{G}^{V},\delta_x)$, if
\begin{enumerate}
\item[\textbf{1.)}] $\p(X_{0}\equiv x)=1$ (i.e. $\p^{X_{0}}=\delta_{x}$), and
\item[\textbf{2.)}] for all $f\in C_{T}\calC^{\epsilon}$ with $\varepsilon > 0$ and for all $u^{T}\in\mathcal{C}^{3}$, the process $M=(M_{t})_{t\in [0,T]}$ is a martingale under $\p$ with respect to $(\mathcal{F}_{t})$, where
\begin{align}
M_{t}=u(t,X_{t})-u(0,x)-\int_{0}^{t}f(s,X_{s})ds
\end{align} and where $u$ solves the Kolmogorov backward equation $\mathcal{G}^{V}u=f$ with terminal condition $u(T,\cdot)=u^{T}$.
\end{enumerate} 
\end{definition}

This is a generalization of the classical notion of a weak solution, in the sense that if $V^{n}$ is a bounded and measurable function, then $(X^{n}_{t})_{t\in [0,T]}$ is a weak solution to 
\begin{align}\label{eq:Xn}
	dX_{t}^{n}=V^{n}(t,X_{t}^{n})dt+dL_{t}, \qquad X_{0}^{n}=x,
\end{align}
if and only if it solves the martingale problem of~\cref{def:martp}.

Our main result is:
\begin{theorem}\label{thm:mainthm}
Let $\alpha\in (1,2]$ and $L$ be a symmetric, $\alpha$-stable Lévy process, such that the measure $\nu$ satisfies \cref{ass}. Let $T>0$ and $\beta\in ((2-2\alpha)/3,0)$ and let $V\in \mathcal X^\beta$ be as in \cref{def:roughdist}. Then for all $x\in\R^{d}$ there exists a unique solution $\mathbb{Q}$ on $(\Omega,\mathcal F)$ of the martingale problem for $(\mathcal{G}^{V},\delta_x)$. Under $\mathbb{Q}$ the canonical process is a strong Markov process.
\end{theorem}

To prove the theorem, we first establish some auxiliary results.

\begin{lemma}\label{lem:moments}
Let $\alpha \in (1,2)$ and let $\pi$ be the Poisson random measure of the $\alpha$-stable Lévy process $L$. We define for a multi-index $\omega \in \mathbb{N}_0^n$ with $n \in
  \mathbb{N}$:
  \[ | \omega | := \omega_1 + 2 \omega_2 + \cdots + n \omega_n . \]
  Then we have for all $C>0$ and $t>r$:
  \[ \mathbb{E} \left[ \left( \int_r^t \int_{| y | \leqslant C} | y |^2 \pi (d
     s, d y) \right)^n \right] \lesssim \sum_{\omega \in \mathbb{N}^n_0 : |
     \omega | = n} \prod_{i = 1}^n \left( (t - r) \int_{| y | \leqslant C} | y
     |^{2 i} \mu (d y) \right)^{\omega_i} . \]
\end{lemma}

We give the proof in~\cref{sec:campbell-appendix}.

\begin{lemma}\label{lem:pimart}
Let $\alpha \in (1,2)$, let $\theta\in (1,\alpha)$ and $u\in C_{T}\calC^{\theta}\cap C_{T}^{\theta/\alpha}L^{\infty}$, and let $\rho\in 2\N$. Let moreover $\hat{\pi}$ be the compensated Poisson random measure of the $\alpha$-stable Lévy process $L$. Then we have uniformly in $0 \leqslant r \leqslant t \leqslant T$:
\begin{align*}
\E\bigg[\abs[\bigg]{\int_{r}^{t} \int_{\R^{d}}\paren[\big]{(u(t,X_{s-}+y)-u(t,X_{s-}))-(u(s,X_{s-}+y)-u(s,X_{s-}))}\hat{\pi}(d s,d y)}^\rho\bigg]\lesssim\abs{t-r}^{\rho\theta/\alpha}.
\end{align*}
\end{lemma} 

\begin{proof}
To abbreviate the notation we write $\Delta_y u(s,x):= u(s,x+y) - u(s,x)$. By the Burkholder-Davis-Gundy inequality together with \cite[Lemma 8.21]{peszat_zabczyk_2007} we get for
any $\rho \geqslant 1$ and for $C > 0$ to be chosen later
\begin{align}\label{eq:piest}
  \MoveEqLeft[2] \mathbb{E} \left[ \left| \int_r^t \int_{\mathbb{R}^d} (\Delta_y u (t, X_{s -}) - \Delta_y u (s, X_{s -})) \hat{\pi} (d
  s, d y) \right|^{\rho} \right]\nonumber\\
  & \lesssim \mathbb{E} \left[ \left| \int_r^t \int_{\mathbb{R}^d} (\Delta_y u (t, X_{s -}) - \Delta_y u (s, X_{s -}))^2
  \pi (d s, d y) \right|^{\rho / 2} \right]\nonumber\\
  & \lesssim \mathbb{E} \left[ \left| \int_r^t \int_{| y | \leqslant C} (\Delta_y u (t, X_{s -}) - \Delta_y u (s, X_{s -}))^2 \pi (d s, d y) \right|^{\rho / 2} \right]\nonumber\\ 
  &\quad +\mathbb{E} \left[ \left| \int_r^t \int_{| y | > C} (\Delta_y u (t, X_{s -}) - \Delta_y u (s, X_{s -}))^2 \pi (d s,
  d y) \right|^{\rho / 2} \right] .
\end{align}
Since $\pi$ is a positive measure, the second term on the right hand side is bounded by
\begin{align*}
  \MoveEqLeft[6] \mathbb{E} \left[ \left| \int_r^t \int_{| y | > C} (\Delta_y u (t, X_{s -}) - \Delta_y u (s, X_{s -}))^2 \pi (d s, d y)
  \right|^{\rho / 2} \right]\\
  & \lesssim | t - r |^{\rho\theta/\alpha} \| u
  \|_{C^{\frac{\theta}{\alpha}}_T L^{\infty}}^{\rho} \mathbb{E} \left[ \left|
  \int_r^t \int_{| y | > C} \pi (d s, d y) \right|^{\rho / 2} \right] .
\end{align*}
The integral inside the expectation is a Poisson distributed random variable with parameter $(t
- r)  \mu (\{ y : | y | > C \}) \simeq (t-r) C^{-\alpha}$. This motivates the choice $C= (t - r)^{1 / \alpha}$, for which this term is of the claimed order. For the first term on the right hand side of \eqref{eq:piest}, we estimate by the mean value theorem and using the time regularity of $\nabla u$ (cf. also \cref{rem:nablau}):
\begin{align*}
  \MoveEqLeft[6] \mathbb{E} \left[ \left| \int_r^t \int_{| y | \leqslant C} (\Delta_y u (t, X_{s -}) - \Delta_y u (s, X_{s -}))^2 \pi (d
  s, d y) \right|^{\rho / 2} \right]\\
  & \lesssim | t - r |^{\rho(\theta - 1)/\alpha } \| \nabla u
  \|_{C_T^{(\theta - 1)/\alpha} L^{\infty}}^{\rho} \mathbb{E} \left[
  \left| \int_r^t \int_{| y | \leqslant C} | y |^2 \pi (d s, d y)
  \right|^{\rho / 2} \right] .
\end{align*}
Now  by \cref{lem:moments} and by the choice $C=(t-r)^{1/\alpha}$, we obtain
\[ \mathbb{E} \left[ \left( \int_r^t \int_{| y | \leqslant C} | y |^2 \pi (d
     s, d y) \right)^{\rho/2} \right] \lesssim \sum_{\omega \in \mathbb{N}^n_0 : |
     \omega | = \rho/2} \prod_{i = 1}^{\rho/2} \left( (t - r) \int_{| y | \leqslant C} | y
     |^{2 i} \mu (d y) \right)^{\omega_i}\lesssim\abs{t-r}^{\rho/\alpha}, \]
where we used that $\int_{\abs{y}\leqslant C}\abs{y}^{k}\mu(dy)\simeq C^{k-\alpha}$ for $k\geqslant 2$. Together this yields for any $\rho\in 2\N$
\begin{align*}
\MoveEqLeft[6]
  \mathbb{E} \left[ \left| \int_r^t \int_{\mathbb{R}^d}(\Delta_y u (t, X_{s -}) - \Delta_y u (s, X_{s -})) \hat{\pi} (d
  s, d y) \right|^{\rho} \right]
 \\& \lesssim \abs{t-r}^{\rho\theta/\alpha}+\abs{t-r}^{\rho(\theta-1)/\alpha}\abs{t-r}^{\rho/\alpha}\simeq\abs{t-r}^{\rho\theta/\alpha}.
\end{align*}
\end{proof}

\begin{corollary}\label{lem:hd}
In the setting of Theorem~\ref{thm:mainthm}, let $(V^{n})_{n \in \N} \subset C_{T}C^{\infty}_{b}(\R^d,\R^d)$ be a smooth approximation with $(V^{n},\calK(V^{n}))\to \mathcal{V}$ in $\calX^{\beta}$. Let $(X^{n}_{t})_{t\in [0,T]}$ be the strong solution of the SDE
\begin{align*}
dX_{t}^{n}=V^{n}(t,X^{n}_{t})dt+dL_{t},\qquad X_{0}=x\in \R^{d}.
\end{align*}
Let $\theta\in ((2-\beta)/2,\alpha+\beta)$ and $\rho\in 2\N$. Then, we have uniformly in $n\in \N$, and $0\leqslant r \leqslant t \leqslant T$:
\begin{align}\label{eq:h1}
\sup_{n}\E\bigg[\abs[\bigg]{\int_{r}^{t}V^{n}(s,X^{n}_{s})ds}^{\rho}\bigg]\lesssim\abs{t-r}^{\theta\rho/\alpha}.
\end{align}
\end{corollary}

\begin{proof}
Let $t\in (0,T]$ and consider the solution $u^{n,t}\in C_{t}C^\infty_b(\R^d,\R^d)$ of the system of equations
\begin{align*}
\mathcal{G}^{V^{n}}u^{n,t,i}=V^{n,i},\qquad u^{n,t,i}(t,\cdot)=0,\qquad\text{for } i=1,...,d.
\end{align*}
For $\beta \in (\frac{2-2\alpha}{3},\frac{1-\alpha}{2}]$ this equation is not exactly of the same type as the equation in~\cref{thm:rgeneq}, because we prescribe the terminal condition at time $t\leqslant T$ and not in $T$. We still use the paracontrolled ansatz $u^{n,t}=(u^{n,t})'\para J^{T}(V^n)+(u^{n,t})^{\sharp}$, i.e. we do not replace $J^T V^n$ by $J^t V^n$, because as $n \to \infty$ we only control $\nabla J^T V^n\reso V^n$ but not $\nabla J^t V^n \reso V^n$. This means there is a blowup of $\|(u^{n,t})^\sharp(s) \|_{2\theta-1}$ as $s \to t$. We discuss below how to deal with this singularity, and we will see that
\begin{align}\label{eq:bound}
\sup_{n\in \N, t\in[0,T]}\norm{u^{n,t}}_{C_{t}\calC^{\theta}_{\R^{d}}} + \norm{u^{n,t}}_{C_{t}^{\theta/\alpha}L^{\infty}_{\R^{d}}}<\infty.
\end{align}

Let first $\alpha \in (1,2)$. Then we apply Itô's formula to $u^{n,t}(t,X^{n}_{t})-u^{n,t}(r,X^{n}_{r})$ and we use that $X^{n}$ solves the SDE with drift $V^{n}$ and that $\mathcal{G}^{V^{n}}u^{n}=V^{n}$ to obtain 
\[
\int_{r}^{t}V^{n}(s,X^{n}_{s})ds=u^{n,t}(t,X^{n}_{t})-u^{n,t}(r,X^{n}_{r})+\int_{r}^{t}\int_{\R^{d}}\paren[\big]{u^{n,t}(s,X^{n}_{s-}+y)-u^{n,t}(s,X^{n}_{s-})}\hat{\pi}(ds,dy).
\] 
As $u^{n,t}(t) = 0$ and by~\eqref{eq:bound} we obtain
\begin{align*}
\abs{u^{n,t}(t,X^{n}_{t})-u^{n,t}(r,X^{n}_{r})}=\abs{u^{n,t}(t,X^{n}_{r})-u^{n,t}(r,X^{n}_{r})} \leqslant\abs{t-r}^{\theta/\alpha}\norm{u^{t}}_{C_{t}^{\theta/\alpha}L^{\infty}} \lesssim \abs{t-r}^{\theta/\alpha}.
\end{align*}
Using once more that $u^{n,t}(t) = 0$, we obtain from~\cref{lem:pimart}:
\begin{align*}
&\E\bigg[\abs[\bigg]{\int_{r}^{t}\int_{\R^{d}\setminus\{0\}}\paren[\big]{u^{n,t}(s,X^{n}_{s-}+y)-u^{n,t}(s,X^{n}_{s-})}\hat{\pi}(ds,dy)}^{\rho}\bigg]\nonumber\\&
=\squeeze[1]{\E\bigg[\abs[\bigg]{\int_{r}^{t}\int_{\R^{d}\setminus\{0\}}\paren[\big]{(u^{n,t}(s,X^{n}_{s-}+y)-u^{n,t}(s,X^{n}_{s-})-(u^{n,t}(t,X^{n}_{s-}+y)-u^{n,t}(t,X^{n}_{s-}))}\hat{\pi}(ds,dy)}^{\rho}\bigg]}\\&\lesssim\abs{t-r}^{\theta\rho/\alpha},
\end{align*}
so~\eqref{eq:h1} holds for $\alpha \in (1,2)$. For $\alpha=2$ the argument is essentially the same, except much easier: Then we only have to replace the jump martingale $\int_{r}^{t}\int_{\R^{d}}\paren[\big]{u^{n,t}(s,X^{n}_{s-}+y)-u^{n,t}(s,X^{n}_{s-})}\hat{\pi}(ds,dy)$ by $\int_{r}^{t}\nabla u^{n,t}(s,X_{s}^{n})dB_{s}$ and apply the Burkholder-Davis-Gundy inequality.

Therefore, the proof is complete once we show~\eqref{eq:bound}. For that purpose we introduce the singular spaces
\begin{align*}
	\mathcal{M}_{t}^{\sigma}\calC^{\gamma}:=\{f\in C([0,t],\mathcal{S}')\mid s\mapsto(t-s)^{\sigma}f(s)\in C_{t}\calC^{\gamma}\},
\end{align*} 
and we adapt the definition of paracontrolled distributions by requiring  $u^{n,t} = (u^{n,t})'\para J^{T}(V^n)+(u^{n,t})^{\sharp}$, $u^{n,t}\in C_{t}\calC^{\theta}\cap C_{t}^{\theta/\alpha}L^{\infty}$, $(u^{n,t})'\in C_{t}\calC^{\theta}$, $(u^{n,t})^{\sharp}\in C_{t}\calC^{\theta}\cap\mathcal{M}_{t}^{(\theta-1)/\alpha}\calC^{2\theta-1}$. Since the blow-up $(\theta-1)/\alpha$ is less than $1$, we can then use techniques for paracontrolled distributions with such singularities (see e.g.~\cite[Section~6]{Gubinelli2017KPZ}) to see that the paracontrolled norm of $u^{n,t}$ is bounded in $n$ and $t$, so in particular~\eqref{eq:bound} holds.
\end{proof}

\begin{proof}[Proof of \cref{thm:mainthm}]
Let $(V^{n})_{n \in \N} \subset C_{T}C^{\infty}_{b}(\R^d,\R^d)$ be such that $(V^{n},\calK(V^{n}))\to \mathcal{V}$ in $\calX^{\beta}$ and let $X^n$ be the unique strong solution of the SDE
\begin{align}
dX^{n}_{t}=V^{n}(t,X^{n}_{t})dt+dL_{t},\qquad X^{n}_{0}=x.
\end{align}

To prove the existence of a solution to the martingale problem for $(\mathcal{G}^{V},\delta_{x})$ we follow the usual strategy: We show tightness of $(X^n)_{n \in \N}$, and then we show that every limit point solves the martingale problem for $(\mathcal{G}^{V},\delta_{x})$. Then we show that the solution to that martingale problem is unique in law, and therefore $(X^n)$ converges weakly.

\textit{Step 1: Tightness of $(\p^{X^{n}})$ on $D([0,T],\R^{d})$}.\\
We apply \eqref{eq:h1} from \cref{lem:hd} for $\rho\in 2\N$ large enough so that $\theta\rho/\alpha > 1$, which shows that the drift term $A^{n}:=\int_{0}^{\cdot}V^{n}(s,X_{s}^{n})ds$ satisfies Kolmogorov's tightness criterion. Therefore, $(A^{n})$ is tight in $C([0,T],\R^{d})$ and thus in particular $C$-tight in $D([0,T],\R^{d})$ (meaning that every limit point is continuous). By \cite[Corollary VI.3.33]{shiryaev}, we thus obtain the tightness of the tuple $(A^{n},L)$ and of $X^{n}=x+A^ {n}+L$.

\textit{Step 2: Any weak limit solves the martingale problem for $(\mathcal{G}^{V},\delta_{x})$}.\\
We consider a weakly convergent subsequence, also denoted by $(\p^{X^{n}})$, and we write $\Q$ for its limit. Let $X$ be the canonical process on $D([0,T],\R^d)$ and let $\E_{n}[\cdot]$ (resp. $\E_\Q[\cdot]$) denote integration w.r.t.\ $\p^{X^{n}}$ (resp. $\mathbb{Q}$). Let $f\in C_{T}\calC^{\epsilon}$ and $u^{T}\in\calC^{3}$, and let $(f^n)_{n \in \N} \subset C_T C^\infty_b$ be such that $f^n$ converges to $f$ in $C_T \calC^\varepsilon$. Let $u^{n}$ be the solution of $\mathcal{G}^{V^{n}}u^{n}=f^n$ with terminal condition $u^{n}(T,\cdot)=u^{T}$. Since $f^n$ and $V^n$ are smooth we have $u^n \in C^{1,2}([0,T]\times \R^d)$ and $u^n$ is a strong solution of the Kolmogorov backward equation. We can thus apply Ito's formula for c\`adl\`ag processes to $u^{n}(t,X_{t})$ under the measure $\p^{X_n}$ and obtain as the operators $-\La$ and $A$ from \eqref{eq:functional} agree on $C^{\infty}_{b}$ (and in fact $u^{n}_{t}\in C^{\infty}_{b}$), that in the jump case $\alpha\in (1,2)$
\begin{align*}
M^{n}_{t}:=& \ u^{n}(t,X_{t})-u^{n}(0,x)-\int_{0}^{t}f^n(s,X_{s})ds\\
=&\ u^{n}(t,X_{t})-u^{n}(0,x)-\int_{0}^{t}\mathcal{G}^{V^{n}}u^{n}(s,X_{s})ds\\
=&\ \int_{0}^{t}\int_{\R^{d}}\paren[\big]{u^{n}(r,X_{r-}+y)-u^{n}(r,X_{r-})}\hat{\pi}(dr,dy)
\end{align*} 
is a martingale in the canonical filtration. Indeed, $M^n$ is a local martingale because it is a stochastic integral against a compensated Poisson random measure, and it is a true martingale because $u^{n}(s,X_{s-}+y)-u^{n}(s,X_{s-})$ is square-integrable w.r.t. $\p\otimes dr\otimes\mu$, where we use the boundedness of $u^{n}$ for the big jump part and the boundedness of $\nabla u^{n}$ for the small jump part. In the Brownian case ($\alpha=2$) we have $M^{n}=\int_{0}^{\cdot}\nabla u^{n}(s,X_{s}^{n})dB_{s}$, which is a martingale because $\nabla u^n$ is bounded.

Let now $u$ be the solution to $\mathcal G^V u =f$ with terminal condition $u(T) = u^T$. By the continuity of the solution map, $(u^{n})$ converges to $u$ in the spaces $C_{T}\calC^{\theta}$ and $C_{T}^{\theta/\alpha}L^{\infty}$, for $\theta\in ((2-\beta)/\beta,\alpha+\beta)$. We show that $(M_{t})_{t\in [0,T]}$ is a martingale under $\mathbb{Q}$, where
\begin{align}\label{eq:show}
M_{t}=u(t,X_{t})-u(0,x)-\int_{0}^{t}f(s,X_{s})ds.
\end{align}
For that purpose let $0\leqslant r\leqslant t \leqslant T$ and let $F:D([0,r],\R^{d})\to\R$ be continuous and bounded. Since $M^{n}$ is a martingale under $\p^{X^{n}}$, we have
\begin{align}\label{eq:weknow}
\E_{n}[(M^{n}_{t}-M^{n}_{r})F((X_{u})_{u\leqslant r})]=0.
\end{align}
We define for $x\in D:=D([0,T],\R^d)$
\begin{align*}
M^{n}_{r,t}(x):=\paren[\bigg]{u^{n}(t,x(t))-u^{n}(r,x(r))-\int_{r}^{t}f^n(u,x(u))du},
\end{align*} and $M_{r,t}(x)$ analogously with $u^{n}, f^n$ replaced by $u,f$. We further define $M^{n}_{0,t}(x)=:M^{n}_{t}(x)$ and $M_{0,t}(x)=:M_{t}(x)$. We want to let $n\to\infty$ in \eqref{eq:weknow}. Therefore, we first note that $\sup_{x\in D}\abs{M^{n}_{t}(x)-M_{t}(x)}\to 0$ for $n\to\infty$, by the convergence of $(u^{n},f^n)$ to $(u,f)$ in $C_{T}C_b \times C_T C_b \subset C_T \calC^\theta \times C_T \calC^\varepsilon$. Thus, we obtain, by boundedness of $F$, that 
\begin{align*}
\lim_{n\to\infty}\E_{n}[M_{r,t}F((X_{u})_{u\leqslant r})]=0.
\end{align*}
Now, by \cite[Proposition VI.2.1]{shiryaev}, we know that the map $D\ni x\mapsto \int_{0}^{t}f(s,x(s))ds$ is continuous w.r.t. the $J_{1}$-topology and it is bounded by boundedness of $f$. Moreover, if we know that $\mathbb{Q}(\Delta X_{t}=\Delta X_{r}=0)=1$, then by \cite[Proposition VI.3.14]{shiryaev} and since $X^{n}\to X$ in distribution in $D$, we have that $X^{n}_{t}\to X_{t}$ and $X^{n}_{r}\to X_{r}$ in distribution. Together this gives (as $\R\ni y\mapsto u(t,y)-u(r,y)$ is continuous and bounded) 
\begin{align*}
0=\lim_{n\to\infty}\E_{n}[M_{t,r}F((X_{u})_{u\in [0,T]})]=\E_{\mathbb{Q}}[M_{t,r}F((X_{u})_{u\in [0,T]})],
\end{align*}
and since $0\leqslant r \leqslant t \leqslant T$ and $F$ were arbitrary, we obtain that $\mathbb{Q}$ solves the martingale problem for $(\mathcal{G}^{V},\delta_{x})$. So it remains to show that indeed $\Q(\Delta X_t = \Delta X_r = 0)$. Since  the map $C([0,T],\R^d) \times D\ni(x,y)\mapsto x+y\in D$ is continuous by \cite[Section VI.1b, Proposition VI.1.23]{shiryaev} and since $(A^{n},L)$ is tight by Step~1, we obtain (possibly along a further subsequence)
\begin{align*}
X\leftarrow X^{n}= x + \int_{0}^{\cdot}V^{n}(s,X^{n}_{s})ds+L\to x + A+L\quad\text{ in distribution in $D$,}
\end{align*} where $A$ denotes the continuous limit of the drift term. Therefore, $\mathbb{Q}(\Delta X_{t}=\Delta X_{r}=0)=\p(\Delta L_{t}=\Delta L_{r}=0)=1$, and this shows that $\Q$ indeed solves the martingale problem for $(\mathcal G^V,\delta_x)$.

\textit{Step 3: Uniqueness for the martingale problem and strong Markov property}.\\
Let $\Q_{1}$ and $\Q_{2}$ be two solutions of the martingale problem for $\calG^{V}$ with the same initial distribution $\mu=\Q_{1}^{X_{0}}=\Q_{2}^{X_{0}}$. Let $f\in C_{T}\calC^{\epsilon}$ and let $u$ be the solution of $\mathcal G^V u = f$, $u(T) = 0$. Then we obtain for $i=1,2$,
\begin{align*}
\int_{\R^{d}} u(0,x)\mu(dx)=\E_{\Q_{i}}\brackets[\bigg]{u(T,X_{T})-\int_{0}^{T}f(s,X_{s})ds}=-\E_{\Q_{i}}\brackets[\bigg]{\int_{0}^{T}f(s,X_{s})ds}.
\end{align*} Thus, we have for all $f\in C_{T}\calC^{\epsilon}$
\begin{align*}
\E_{\Q_{1}}\brackets[\bigg]{\int_{0}^{T}f(s,X_{s})ds}=\E_{\Q_{2}}\brackets[\bigg]{\int_{0}^{T}f(s,X_{s})ds}.
\end{align*}
Therefore, $\Q_{1}^{X_{t}}=\Q_{2}^{X_{t}}$ for all $t\in [0,T]$, that is, the one dimensional marginal distributions of $\Q_1$ and $\Q_2$ agree. Indeed, this follows by taking $f_{\delta}(s,x)=\delta^{-1}h_{\delta}(s)g(x)$ for $h_{\delta}\simeq\mathbf{1}_{[t,t+\delta]}$ and $g\in\calC^{\epsilon}$ and letting $\delta\to 0$. Now \cite[Theorem 4.4.3]{ethier} shows that $\Q_1 = \Q_2$ and that under the solution $\Q$ to the martingale problem for $(\mathcal G^V,\delta_x)$ the canonical process is a strong Markov process.
\end{proof}
\end{section}

\begin{section}{Brox diffusion with Lévy noise}\label{sect:bd}
The Brox diffusion is the solution $X$ of the SDE
\begin{align}
dX_{t}=\dot{W}(X_{t})dt+dB_{t},\qquad X_{0}=x\in\R,
\end{align} where $B$ is a standard Brownian motion and $(W(x))_{x \in \R}$ is a two-sided standard Brownian motion that is independent of $B$. This model was introduced by Brox~\cite{Brox1986} as a continuous analogue of Sinai’s random walk, with the motivation that when studying $X$ we can exploit the scaling properties of $W$ and $B$. Brox's construction is based on time and space transformations as in the It\^o-McKean construction of diffusions. It is natural to replace $W$ or $B$ by $\alpha$-stable Lévy processes, which also have nice scaling properties. The construction of the process with $W$ replaced by a L\'evy process is not much of a problem, as the It\^o-McKean approach still works~\cite{Tanaka1986, Carmona1997, Kusuoka2017}. On the other hand, replacing $B$ by an $\alpha$-stable L\'evy process is more delicate and it is not obvious if the Ito-McKean construction could work. But using our approach we can hope to solve the martingale problem for the SDE
\begin{align}
dX_{t}=\dot{W}(X_{t})dt+dL_{t},\qquad X_{0}=x\in\R.
\end{align}
To be precise, the white noise $\dot{W}$ is not actually an element of any Besov space, but only of weighted Besov spaces: With $\langle x \rangle = (1+|x|^2)^{1/2}$ we have $\langle \cdot \rangle^{-\kappa} \dot{W}  \in \calC^{-1/2-}$ for all $\kappa > 0$. It is possible to extend our analysis of the martingale problem to allow for a drift term in a suitable weighted Besov space, and at the end of this section we discuss how this could be done. But to simplify the presentation  we consider a periodic white noise $\dot{W}$ instead, which is in the unweighted space $\calC^{-1/2-}$. Note that this regularity is not in the Young regime, no matter which $\alpha \in (1,2]$ we choose, and therefore the methods of~\cite{Athreya2018, deRaynal2019} do not apply and we are not aware of any other way of constructing $X$, apart from the approach we present here.

So let $\xi=\dot{W}$ be a $1$-periodic white noise, that is, $\xi$ is a centered Gaussian process with values in $\mathcal S'(\mathbb T)$, where $\mathbb{T} = \R/\mathbb Z$ is the one-dimensional torus and $\mathcal S'(\mathbb T)$ is the space of Schwartz distributions on $\mathbb T$, i.e. the topological dual of $C^\infty(\mathbb T)$. The covariance of $\xi$ is $\E[\xi(\varphi)\xi(\psi)]=\langle \varphi,\psi\rangle_{L^{2}(\mathbb{T})}$ for $\varphi, \psi \in C^\infty(\mathbb T)$.
To any $u \in \mathcal S'(\mathbb T)$ we associate a periodic distribution on the real line by setting $u^{\R}(\varphi)=u(\sum_{k\in\mathbb{Z}}\varphi(\cdot+k))$, $\varphi\in\mathcal{S}$. If $u\in\calC^{\beta}(\mathbb{T})$, then $u^{\R}\in\calC^{\beta}$. Here $\calC^\beta(\mathbb T)$ is a Besov space on the torus, which is defined in the same way as on the real line, except using the Fourier transform on $\mathbb{T}$ and inverse Fourier transform on $\mathbb Z^d$.

We choose $\xi$ independently of the L\'evy process $L$, and we consider a fixed ``typical'' realization $\xi(\omega)$. To apply the theory that we developed in this paper, we need to construct a canonical enhancement of $\xi(\omega)^R$ in such a way that we obtain an enhanced drift in the sense of~\cref{def:roughdist}.

We first note that almost surely $\xi\in\CTcalC^{-1/2-}(\mathbb{T})$ (so we let $\beta=-1/2-\epsilon$ for some very small $\epsilon>0$), see e.g.~\cite[Exercise~11]{Gubinelli2015EBP}. Therefore, $\xi(\omega)^\R \in \CTcalC^{-1/2-}$ for almost all $\omega$. It remains to construct $(J^{T}(\nabla \xi)\reso \xi)(\omega )\in\CTcalC^{(-2+\alpha)-}(\mathbb{T})$ for almost all $\omega$, which we will do in the next lemma.
\begin{lemma}\label{lem:wnreg}
Let $\alpha\in (3/2,2]$, $\vartheta<\alpha-2$ and $J^{T}(u)(t)=\int_{t}^{T}P_{r-t}u(r)dr$, for the semigroup $(P_{t})$ generated by $\La$, $P_{t}\phi=\F^{-1}(e^{-t\psi^{\alpha}_{\nu}}\F\phi)$. Let $\xi^{n}=\sum_{\abs{k}\leqslant n}\hat{\xi}(k)e_{k}$, where $(e_{k})_{k\in\mathbb{Z}}=(e^{-2\pi i k\cdot})_{k\in\mathbb{Z}}$ is the Fourier basis of $L^{2}(\mathbb{T})$. Then $(J^{T}(\nabla \xi^{n})\reso \xi^{n})_{n}$ converges in probability in $C_{T}\calC^{\vartheta}(\mathbb T)$ to a limit denoted by $J^{T}(\nabla\xi)\reso\xi\in C_{T}\calC^{\vartheta}(\mathbb T)$.
\end{lemma}
\begin{proof}
We carry out the computations for $n=\infty$ and show that $J^{T}(\nabla\xi)\reso\xi\in C_{T}\calC^{\vartheta}(\mathbb T)$ can be constructed as a random variable in the second Wiener-It\^o chaos generated by $\xi$. Since the kernel appearing in the definition of $J^{T}(\nabla\xi)\reso\xi$ provides a uniform bound for the kernels that appear in the chaos representation of $(J^{T}(\nabla \xi^{n})\reso \xi^{n})_{n}$, the claimed convergence then follows from the dominated convergence theorem.

To bound $J^{T}(\nabla\xi)\reso\xi$, note that $J^{T}(\nabla\xi)(t)=\varrho_{t}\ast\xi$, where $\varrho_{t}= \nabla \F^{-1}(\int_{t}^{T}e^{(r-t)\psi^{\alpha}_{\nu}}dr)$. We first derive a bound on the expectation of the $B_{p,p}^{\zeta}$-norm (for $\zeta$ to be chosen afterwards) of the increment $(\varrho_{t}\ast\xi)\reso\xi-(\varrho_{s}\ast\xi)\reso\xi=((\varrho_{t}-\varrho_{s})\ast\xi)\reso\xi$. Using this bound, our claim will follow from the Besov embedding theorem together with Kolmogorov's continuity criterion. We have
\begin{align*}
\E[\norm{(\varrho_{t}-\varrho_{s})\ast\xi\reso\xi}_{B_{p,p}^{\zeta}}^{p}]&=\E\big[\sum_{j}2^{j\zeta p}\norm{\Delta_{j}((\varrho_{t}-\varrho_{s})\ast\xi\reso\xi)}_{L^{p}}^{p}\big]\\&=\sum_{j}2^{j\zeta p}\int_{\mathbb{T}}\E[\abs{\Delta_{j}((\varrho_{t}-\varrho_{s})\ast\xi\reso\xi)(x)}^{p}]dx\\&\lesssim\sum_{j}2^{j\zeta p}\int_{\mathbb{T}}\E[\abs{\Delta_{j}((\varrho_{t}-\varrho_{s})\ast\xi\reso\xi)(x)}^{2}]^{p/2}dx,
\end{align*} where in the last step we used that the random variable $\Delta_{j}((\varrho_{t}-\varrho_{s})\ast\xi\reso\xi)(x)$ is in the second (inhomogeneous) Wiener-It\^o chaos and therefore all its moments are comparable by Gaussian hypercontractivity~\cite[Theorem~5.10]{Janson1997}. It remains to estimate 
\begin{align}\label{eq:est1}
\E[\abs{\Delta_{j}((\varrho_{t}-\varrho_{s})\ast\xi\reso\xi)(x)}^{2}]=\E[\abs{((\varrho_{t}-\varrho_{s})\ast\xi\reso\xi)(\kappa_{j}(x-\cdot))}^{2}],
\end{align} where $\kappa_{j}=\F^{-1}_{\mathbb Z} p_{j} = \sum_{k \in \mathbb Z} e^{2\pi i k \cdot} p_j(k)$. Let now $\psi_{\reso}(x,y)=\sum_{\abs{l_1-l_2}\leqslant 1}\kappa_{l_1}(x)\kappa_{l_2}(y)$. Then with formal notation:
\begin{align*}
(\varrho_{t}-\varrho_{s})\ast\xi\reso\xi(x)=\iint\psi_{\reso}(x-y_{1},x-y_{2})((\varrho_{t}-\varrho_{s})\ast\xi)(y_{1})\xi(y_{2})dy_{1}dy_{2},
\end{align*}
and thus
\begin{align*}
\squeeze[1]{(\varrho_{t}-\varrho_{s})\ast\xi\reso\xi(\kappa_{j}(x-\cdot))=\iiint\kappa_{j}(x-z)\psi_{\reso}(z-y_{1},z-y_{2})\xi((\varrho_{t}-\varrho_{s})(y_{1}-\cdot))\xi(\delta(y_{2}-\cdot))dy_{1}dy_{2}dz.}
\end{align*}
To derive the chaos decomposition of the right hand side, we introduce the kernel
\begin{align*}
A_{j}^{t,s}(x,r_{1},r_{2})=\iiint\kappa_{j}(x-z)\psi_{\reso}(z-y_{1},z-y_{2})(\varrho_{t}-\varrho_{s})(y_{1}-r_{1})\delta(y_{2}-r_{2})dy_{1}dy_{2}dz,
\end{align*}
with which
\begin{align}\label{eq:limit}
(\varrho_{t}-\varrho_{s})\ast\xi\reso\xi(\kappa_{j}(x-\cdot))=W_{2}(A_{j}^{t,s}(x,\cdot,\cdot))+\E[(\varrho_{t}-\varrho_{s})\ast\xi\reso\xi(\kappa_{j}(x-\cdot))],
\end{align} where $W_{2}$ denotes a second order Wiener-Itô integral. We start by estimating the first term on the right hand side: Using the symmetrization $\tilde{A}_j^{t,s}(x,r_1,r_2) = \tfrac12 (A_{j}^{t,s}(x,r_{1},r_{2}) + A_{j}^{t,s}(x,r_{2},r_{1}))$, we have
\begin{align}\label{eq:est2} \nonumber
\E[\abs{W_{2}(A_{j}^{t,s}(x,\cdot,\cdot))}^{2}] & = 2 \norm{\tilde{A}_{j}^{t,s}(x,\cdot,\cdot)}_{L^{2}(\mathbb{T}^{2})}^{2} \leqslant 2 \norm{A_{j}^{t,s}(x,\cdot,\cdot)}_{L^{2}(\mathbb{T}^{2})}^{2}  \\
& =\sum_{k_{1},k_{2}\in\mathbb{Z}}\abs[\bigg]{\iint A_{j}^{t,s}(x,r_{1},r_{2})e^{-2\pi i(k_{1}r_{1}+k_{2}r_{2})}dr_{1}dr_{2}}^{2},
\end{align}
where the last equality is Parseval's identity. Now, we obtain by computing each integral iteratively 
\begin{align*}
\squeeze[1]{\iint A_{j}(x,r_{1},r_{2})e^{-2\pi i(k_{1}r_{1}+k_{2}r_{2})}dr_{1}dr_{2}=\hat{\kappa}_{j}(-(k_{1}+k_{2}))e^{-2\pi i(k_{1}+k_{2})x}\hat{\psi_{\reso}}(-k_{1},-k_{2})\widehat{(\varrho_{t}-\varrho_{s})}(-k_{1})},
\end{align*}
where $\hat f(k) = \int_{\mathbb T} f(x) e^{-2\pi i k x} dx$ is the Fourier transform on the torus and
\begin{align*}
\hat{\psi_{\reso}}(k_{1},k_{2}):=\iint\psi_{\reso}(y_{1},y_{2})e^{-2\pi i(k_{1}y_{1}+k_{2}y_{2})}dy_{1}dy_{2}=\sum_{\abs{l_{1}-l_{2}}\leqslant 1} p_{l_{1}}(k_{1})p_{l_{2}}(k_{2}).
\end{align*}
As $\abs{\psi^{\alpha}_{\nu}(k)}\geqslant\abs{k}^{\alpha}$ and $1-e^{-x}\leqslant x^{\epsilon}$ for $x\geqslant 0$, $\epsilon\in [0,1]$, we have for $s<t$ and $\epsilon\in [0,1]$
\begin{align*}
\MoveEqLeft
\abs{\widehat{(\varrho_{t}-\varrho_{s})}(k)}\\&\lesssim\abs{k}\abs[\bigg]{\int_{s}^{t}e^{-(r-s)\psi^{\alpha}_{\nu}(k)}dr+\int_{t}^{T}e^{-(r-t)\psi^{\alpha}_{\nu}(k)}(1-e^{-(t-s)\psi^{\alpha}_{\nu}(k)})dr}\\&\lesssim_{T}\abs{t-s}^{\epsilon}\abs{k}^{1-\alpha+\alpha\epsilon}.
\end{align*}
This leads to
\begin{align*}
\MoveEqLeft
\abs[\bigg]{\iint A_{j}^{t,s}(x,r_{1},r_{2})e^{-2\pi i(k_{1}r_{1}+k_{2}r_{2})}dr_{1}dr_{2}}^{2}\\&\lesssim\abs{t-s}^{2\epsilon}\abs{p_{j}(k_{1}+k_{2})}^{2}\abs[\big]{\sum_{\abs{l_{1}-l_{2}}\leqslant 1}\!\!p_{l_{1}}(k_{1})p_{l_{2}}(k_{2})}^{2}\abs{k_{1}}^{2-2\alpha(1-\epsilon)}.
\end{align*}
Let now $\tilde{p}_{l_{1}}:=\sum_{l:\abs{l-l_{1}}\leqslant 1}p_{l}$. Since for fixed $k_{1}$ there are at most three $l_{1}$ with $p_{l_{1}}(k_{1})\neq 0$, we can bound $\abs[\big]{\sum_{l_1} p_{l_{1}}(k_{1})\tilde{p}_{l_{1}}(k_{2})}^{2}\lesssim  \sum_{l_1} p_{l_{1}}(k_{1})^2\tilde{p}_{l_{1}}(k_{2})^2$ and thus we obtain in \eqref{eq:est2}
\begin{align}\label{eq:k}
\E[\abs{W_{2}(A_{j}^{t,s}(x,\cdot,\cdot))}^{2}]&\lesssim \abs{t-s}^{2\epsilon}\sum_{k_{1},k_{2}}\sum_{l_{1}}p_{j}(k_{1}+k_{2})^{2}p_{l_{1}}(k_{1})^{2}\tilde{p}_{l_{1}}(k_{2})^{2}\abs{k_{1}}^{2-2\alpha(1-\epsilon)}\nonumber\\&=\abs{t-s}^{2\epsilon}\sum_{l_{1}: 2^{j}\lesssim 2^{l_{1}}}\sum_{k_{1}}2^{j}p_{l_{1}}(k_{1})^{2}\abs{k_{1}}^{2-2\alpha(1-\epsilon)}\\&\lesssim\abs{t-s}^{2\epsilon} \sum_{l_{1}: 2^{j}\lesssim 2^{l_{1}}}2^{j}2^{l_{1}}2^{l_{1}(2-2\alpha(1-\epsilon))}\lesssim\abs{t-s}^{2\epsilon} 2^{j(4-2\alpha(1-\epsilon))}\nonumber,
\end{align}
where we used that $p_{i}(k)\neq 0$ for $O(2^{i})$ values of $k$, with $i=j$ respectively $i=l_1$, and we choose $\epsilon \in (0,1)$ so that $3-2\alpha(1-\varepsilon)<0$ to obtain the convergence of the series in the last estimate (recall that we assume $\alpha>3/2$).

Let $e_k(x) = e^{2\pi ikx}$ so that $\int e_k(x) e_l(x) dx = \delta_{k=-l}$. Then the second term on the right hand side of~\eqref{eq:limit} is
\begin{align*}
\MoveEqLeft
\E[(\varrho_{t}-\varrho_{s})\ast\xi\reso\xi(\kappa_{j}(x-\cdot))]^{2}\\&=\squeeze[1]{\paren[\bigg]{\iiint \kappa_{j}(x-z)\psi_{\reso}(z-y_{1},z-y_{2})(\varrho_{t}-\varrho_{s})(y_{1}-y_{2})dy_{1}dy_{2}dz}^{2}}\\
&=\squeeze[1]{\paren[\bigg]{\sum_{k,l,k',l'}\hat{\kappa_{j}}(k)\hat{\psi_{\reso}}(k',l')\widehat{(\varrho_{t}-\varrho_{s})}(l)\iiint e_{k}(x-z)e_{k'}(z-y_{1})e_{l'}(z-y_{2})e_{l}(y_{1}-y_{2})dy_{1}dy_{2}dz}^{2}}\\
&=\paren[\bigg]{\sum_{k'}\hat{\kappa_{j}}(0)\hat{\psi_{\reso}}(k',-k')\widehat{(\varrho_{t}-\varrho_{s})}(k')}^{2}\\
&\lesssim\delta_{j=-1}\abs{t-s}^{2\epsilon}\paren[\bigg]{\sum_{k'}\hat{\psi_{\reso}}(k',-k')\abs{k'}^{1-\alpha(1-\epsilon)}}^{2}\\&\lesssim\delta_{j=-1}\abs{t-s}^{2\epsilon}\sum_{l}2^{l}2^{l(2-2\alpha(1-\epsilon))} \lesssim \delta_{j=-1}\abs{t-s}^{2\epsilon},
\end{align*}
by orthogonality of the Fourier basis $(e_k)$ and where again we assume that $\epsilon\in(0,1]$ is small enough so that $3-2\alpha(1-\varepsilon) < 0$ to guarantee that the series in $l$ converges.

Combining this estimate with~\eqref{eq:limit} and~\eqref{eq:k}, we get via the Besov embedding theorem that for all $\vartheta' < \alpha - 2$ there exists $\varepsilon>0$ such that for all $p>1$ (by taking $\zeta=\vartheta'$),
\begin{align*}
\E[\norm{(\varrho_{t}-\varrho_{s})\ast\xi\reso\xi}_{\calC^{\vartheta'-1/p}}^{p}]\lesssim\E[\norm{(\varrho_{t}-\varrho_{s})\ast\xi\reso\xi}_{B_{p,p}^{\vartheta'}}^{p}]\lesssim\abs{t-s}^{\epsilon p}.
\end{align*}
Aftering choosing $p$ large enough so that $\varepsilon p >0$ we obtain from Kolmogorov's continuity criterion that $J^T(\nabla \xi)\reso \xi \in C_T \calC^{\vartheta'-1/p}$. Given $\vartheta < \alpha-2$ as in the statement of the theorem, it now suffices to take $\vartheta' \in (\vartheta, \alpha-2)$ and then $p$ large enough so that $\vartheta'-1/p \geqslant \vartheta$.
\end{proof}

By freezing a ``typical'' realization of $\xi(\omega)$, we obtain the following corollary of \cref{lem:wnreg} and \cref{thm:mainthm}.
\begin{theorem}
Let $\alpha\in (7/4,2]$ and let $\xi$ be a periodic white noise on a probability space $(\Omega, \F, p)$. Then for almost all $\omega$ there exists a unique solution to the ``quenched martingale problem''
associated to the Brox diffusion with symmetric, $\alpha$-stable Lévy process $L$,
\begin{align*}
dX_{t}=\xi(\omega)(X_{t})dt+dL_{t},\qquad X_{0}=x\in\R.
\end{align*}
If we denote the distribution of $X$ by $P_\omega$, then the ``annealed measure'' $\int P_\omega(\cdot) \p(d\omega)$ is the distribution of a Brox diffusion in a white noise potential, driven by an independent symmetric $\alpha$-stable L\'evy process $L$.
\end{theorem}

\begin{remark}
	By analogy with rough path regularities, the constraint $\alpha>7/4$ corresponds to an ``$\alpha > 1/3$ condition'' in rough paths, and we expect that it is possible to treat $\alpha \in (3/2, 7/4]$ by considering higher order expansions of the Kolmogorov backward equation. To carry out this analysis we would need to use regularity structures~\cite{Hairer2014} or the higher order paracontrolled calculus of~\cite{Bailleul2019High}. The constraint $\alpha > 3/2$ appears in the construction of the resonant product $J^T(\nabla \xi) \reso \xi$, so it seems to be of a similar nature as the constraint $H>1/4$ for the Hurst index of a fractional Brownian motion that is required to construct its iterated integrals~\cite{Coutin2002}. But in fact not only the probabilistic construction fails at $\alpha=3/2$: At that value the equation is \emph{critical} in the sense of Hairer~\cite{Hairer2014} and we cannot solve it with perturbative techniques such as paracontrolled distributions or regularity structures.
\end{remark}

\begin{remark}
	To avoid dealing with weighted function spaces, we restricted our attention to periodic $\xi$. But we expect that it is also possible to treat the white noise $\xi$ on $\R$ with our approach, at the price of a slightly more involved analysis. In that case we have $\langle \cdot \rangle^{-\kappa} \xi  \in \calC^{-1/2-}$ and $\langle \cdot \rangle^{-\kappa} J^T(\nabla \xi)\reso \xi  \in \calC^{\alpha-2-}$ for all $\kappa > 0$, where $\langle x \rangle = (1+|x|^2)^{1/2}$. With the techniques of~\cite{Delarue2016, Hairer2015Simple, Martin2017} it is still possible to solve the Kolmogorov backward equation for such $\xi$, by working in weighted function spaces with a time-dependent weight. Roughly speaking, if the terminal condition $u_T$ grows like $e^{l |x|^\delta}$ as $x \to \infty$, where $\delta \in (0,1)$ and $l \in R$, then $u(T-t)$ grows like $e^{(l+t)|x|^\delta}$. This might look dangerous because for $\alpha< 2$ our L\'evy noise does not even have finite second moments, let alone finite (sub-)exponential moments. But we can take $l \in \R$ arbitrary, and in particular $l \leqslant -T$ is allowed and then $u(t)$ is bounded for all $t$. In that way it should be possible to extend our results to construct a Brox diffusion with L\'evy noise in a non-periodic white noise potential.
\end{remark}

\end{section}

\appendix

\begin{section}{Appendix}

\begin{subsection}{Commutator estimates}\label{sec:commutator-appendix}
The following commutator estimate between the semigroup generated by $-\La$ and the paraproduct will be used in the proof of \cref{lem:sharp} below.

\begin{lemma}\label{schaudercom}
Let $(P_{t})$ be as in \cref{schauder}.
Then, for $\gamma <1$, $\beta\in\R$ and $\vartheta \geqslant -1$ the following commutator estimate holds:
\begin{align}
\norm{P_{t}(u\para v)-u\para P_{t}v}_{\gamma+\beta+\vartheta}\lesssim t^{-\vartheta/\alpha}\norm{u}_{\gamma}\norm{v}_{\beta}.\label{eq:scom}
\end{align}
\end{lemma}

\begin{proof}
This is \cite[Lemma 5.3.20 and Lemma 5.5.7]{doktorp}, applied to $\varphi(z)=\exp(-\psi^{\alpha}_{\nu}(z))$. 
\end{proof}

\begin{proof}[Proof of \cref{lem:sharp}]
We write $J^{T}( g\para h)(t)- g(t)\para J^{T}(h)(t)=I_{1}(t)+I_{2}(t)$, where 
\begin{align*}
&I_{1}(t)=\int_{t}^{T}(P_{r-t}( g(r)\para h(r))- g(r)\para P_{r-t}h(r))dr,\\
&I_{2}(t)=\int_{t}^{T}( g(r)- g(t))\para P_{r-t}h(r)dr.
\end{align*} For $I_{1}$ we apply \eqref{eq:scom} and obtain for $t\in [T-\overline{T},T]$ as $\sigma<1$ and $-1\leqslant\sigma-\varsigma+1<\alpha$,
\begin{align*}
\norm{I_{1}(t)}_{2\sigma+1}\lesssim\int_{t}^{T}(r-t)^{-\frac{\sigma-\varsigma+1}{\alpha}}\norm{g(r)}_{\calC^{\sigma}_{\R^{d}}}\norm{h(r)}_{\calC^{\varsigma}_{\R^{d}}}dr\lesssim \overline{T}^{\kappa}\norm{g}_{C_{\overline{T},T}\calC^{\sigma}_{\R^{d}}}\norm{h}_{C_{T}\calC^{\varsigma}_{\R^{d}}},
\end{align*} where $\kappa:=1-\frac{\sigma-\varsigma+1}{\alpha}>0$.
Now it follows from the estimates for the paraproduct~\eqref{eq:paraproduct-estimates}, and from the estimate~\eqref{eq:schauder1} for the regularizing effect of $P_t$ as $\sigma>0$ that
\begin{align*}
\norm{I_{2}(t)}_{2\sigma+1}&\lesssim\int_{t}^{T}\norm{g(r)-g(t)}_{L^{\infty}_{\R^{d}}}\norm{P_{r-t}h(r)}_{\calC^{2\sigma+1}_{\R^{d}}}dr\\
&\lesssim \norm{g}_{C^{\sigma/\alpha}_{\overline{T},T}L^{\infty}_{\R^{d}}} \norm{h}_{C_{T}\mathcal{C}^{\varsigma}_{\R^{d}}}  \int_{t}^{T} (r-t)^{\frac{\sigma}{\alpha}}(r-t)^{-\frac{2\sigma+1-\varsigma}{\alpha}}dr\\
&\lesssim \overline{T}^{\kappa}\norm{g}_{C^{\sigma/\alpha}_{\overline{T},T}L^{\infty}_{\R^{d}}}\norm{h}_{C_{T}\mathcal{C}^{\varsigma}_{\R^{d}}},
\end{align*} where $\kappa=1-\frac{\sigma-\varsigma+1}{\alpha}>0$. This is the claimed bound.
\end{proof}
\end{subsection}

\begin{subsection}{An application of Campbell's formula}\label{sec:campbell-appendix}

Here we are in the setting of Lemma~\ref{lem:moments}, i.e. $\pi$ is the Poisson random measure of the $\alpha$-stable Lévy process $L$, $| \omega | := \omega_1 + 2 \omega_2 + \cdots + n \omega_n$, and $C>0$ and $0\leqslant r < t$. Lemma~\ref{lem:moments} follows by plugging $\lambda = 0$ into~\cref{eq:moment-generating} below.

\begin{lemma}
For $\lambda \in \R$ we define the following moment generating function:
  \[ \Phi (\lambda) := \mathbb{E} \left[ \exp \left( \int_r^t \int_{| y |
     \leqslant C} \lambda | y |^2 \pi (d s, d y) \right) \right] . \]
Then the derivatives of $\Phi$ satisfy
  \begin{equation}\label{eq:moment-generating}
	 \Phi^{(n)} (\lambda) = \Phi (\lambda) \sum_{\omega \in \mathbb{N}^n_0 : |
     \omega | = n} c (n, \omega) \prod_{i = 1}^n \left( (t - r) \int_{| y |
     \leqslant C} | y |^{2 i} e^{\lambda | y |^2} \mu (d y) \right)^{\omega_i}
  \end{equation}
  for suitable integers $c (n, \omega)$.
\end{lemma}

\begin{proof}
  We prove this by induction. For $n=0$ the claim is obviously true, so we assume that it holds for $n$ and establish it also for $n+1$. We get with Campbell's formula (see~\cite[Section~3.2]{Kingman1993}):
  \[ \Phi (\lambda) = \exp \left( \int_r^t \int_{| y | \leqslant C}
     (e^{\lambda | y |^2} - 1) \mu (d y) d s \right) = \exp \left( (t - r)
     \int_{| y | \leqslant C} (e^{\lambda | y |^2} - 1) \mu (d y) \right), \]
  and therefore
  \begin{align*}
    \MoveEqLeft \Phi^{(n + 1)} (\lambda)= \partial_{\lambda} \Phi^{(n)} (\lambda)\\
    &  = \partial_{\lambda} \left(
    \Phi (\lambda) \sum_{\omega \in \mathbb{N}^n_0 : | \omega | = n} c (n,
    \omega) \prod_{i = 1}^n \left( (t - r) \int_{| y | \leqslant C} | y |^{2
    i} e^{\lambda | y |^2} \mu (d y) \right)^{\omega_i} \right)\\
    & = \Phi (\lambda) (t - r) \int_{| y | \leqslant C} | y |^{2 }
    e^{\lambda | y |^2} \mu (d y) \sum_{\omega \in \mathbb{N}^n_0 : | \omega |
    = n} c (n, \omega) \prod_{i = 1}^n \left( (t - r) \int_{| y | \leqslant C}
    | y |^{2 i} e^{\lambda | y |^2} \mu (d y) \right)^{\omega_i}\\
    & \quad + \Phi (\lambda) \sum_{\omega \in \mathbb{N}^n_0 : | \omega | = n} c
    (n, \omega) \partial_{\lambda} \left( \prod_{i = 1}^n \left( (t - r)
    \int_{| y | \leqslant C} | y |^{2 i} e^{\lambda | y |^2} \mu (d y)
    \right)^{\omega_i} \right) .
  \end{align*}
  The first term on the right hand side is of the claimed form with
  $\tilde{\omega} = (\omega_1 + 1, \omega_2, \ldots, \omega_n, 0) \in
  \mathbb{N}_0^{n + 1}$ such that $|\tilde \omega|=n+1$. For the second term on the right hand side we get by
  Leibniz's rule
  \begin{align*}
    \MoveEqLeft \partial_{\lambda} \left( \prod_{i = 1}^n \left( (t - r) \int_{| y |
    \leqslant C} | y |^{2 i} e^{\lambda | y |^2} \mu (d y) \right)^{\omega_i}
    \right)\\
    & = \sum_{j = 1}^n \prod_{i \neq j}^n \left( (t - r) \int_{| y |
    \leqslant C} | y |^{2 i} e^{\lambda | y |^2} \mu (d y) \right)^{\omega_i}
    \times \omega_j \left( (t - r) \int_{| y | \leqslant C} | y |^{2 j}
    e^{\lambda | y |^2} \mu (d y) \right)^{\omega_j - 1}\\
    & \hspace{40pt} \times (t - r) \int_{| y | \leqslant C} | y |^{2 (j + 1)}
    e^{\lambda | y |^2} \mu (d y)\\
    & = \sum_{j = 1}^{n + 1} \omega_j \prod_{i = 1}^{n + 1} \left( (t - r)
    \int_{| y | \leqslant C} | y |^{2 i} e^{\lambda | y |^2} \mu (d y)
    \right)^{\tilde{\omega}^j_i},
  \end{align*}
  with $\tilde{\omega}^j_i \in \N_0^{n+1}$ defined by
  \[ \tilde{\omega}^j_i = \left\{ \begin{array}{ll}
       \omega_i, & i \neq j, j + 1,\\
       \omega_j - 1, & i = j,\\
       \omega_{j + 1} + 1, & i = j + 1.
     \end{array} \right. \]
  As required we have $| \tilde{\omega}^j | = | \omega | - j + (j + 1) = | \omega | + 1 = n + 1$, and thus the proof is complete.
\end{proof}

\end{subsection}

\end{section}


\addcontentsline{toc}{chapter}{References}
\begin{thebibliography}{FRW03}

\bibitem[ABM20]{Athreya2018}
Siva Athreya, Oleg Butkovsky, and Leonid Mytnik.
\newblock Strong existence and uniqueness for stable stochastic differential
  equations with distributional drift.
\newblock {\em Ann. Probab.}, 48(1):178--210, 2020.

\bibitem[AKQ14]{Alberts2014}
Tom Alberts, Konstantin Khanin, and Jeremy Quastel.
\newblock The intermediate disorder regime for directed polymers in dimension
  $1+ 1$.
\newblock {\em Ann. Probab.}, 42(3):1212--1256, 2014.

\bibitem[BB19]{Bailleul2019High}
Isma\"{e}l Bailleul and Fr\'{e}d\'{e}ric Bernicot.
\newblock High order paracontrolled calculus.
\newblock {\em Forum Math. Sigma}, 7:e44, 94, 2019.

\bibitem[BC01]{Bass2001}
Richard~F. Bass and Zhen-Qing Chen.
\newblock Stochastic differential equations for {D}irichlet processes.
\newblock {\em Probab. Theory Related Fields}, 121(3):422--446, 2001.

\bibitem[BCD11]{bahouri}
Hajer Bahouri, Jean-Yves Chemin, and Rapha\"el Danchin.
\newblock {\em Fourier Analysis and Nonlinear Partial Differential Equations}.
\newblock Grundlehren der mathematischen Wissenschaften. Springer Berlin
  Heidelberg, 2011.

\bibitem[Bro86]{Brox1986}
Th. Brox.
\newblock A one-dimensional diffusion process in a {W}iener medium.
\newblock {\em Ann. Probab.}, 14(4):1206--1218, 1986.

\bibitem[Car97]{Carmona1997}
Philippe Carmona.
\newblock The mean velocity of a {B}rownian motion in a random {L}\'{e}vy
  potential.
\newblock {\em Ann. Probab.}, 25(4):1774--1788, 1997.

\bibitem[CC18]{Cannizzaro2018}
Giuseppe Cannizzaro and Khalil Chouk.
\newblock Multidimensional {SDE}s with singular drift and universal
  construction of the polymer measure with white noise potential.
\newblock {\em Ann. Probab.}, 46(3):1710--1763, 2018.

\bibitem[CG16]{Catellier2016}
R\'emi Catellier and Massimiliano Gubinelli.
\newblock Averaging along irregular curves and regularisation of {ODE}s.
\newblock {\em Stochastic Process. Appl.}, 126(8):2323--2366, 2016.

\bibitem[CQ02]{Coutin2002}
Laure Coutin and Zhongmin Qian.
\newblock {Stochastic analysis, rough path analysis and fractional Brownian
  motions.}
\newblock {\em Probab. Theory Related Fields}, 122(1):108--140, 2002.

\bibitem[CSZ17]{Caravenna2017}
Francesco Caravenna, Rongfeng Sun, and Nikos Zygouras.
\newblock Polynomial chaos and scaling limits of disordered systems.
\newblock {\em J. Eur. Math. Soc. (JEMS)}, 19(1):1--65, 2017.

\bibitem[DD16]{Delarue2016}
Fran{\c{c}}ois Delarue and Roland Diel.
\newblock Rough paths and 1d {SDE} with a time dependent distributional drift:
  application to polymers.
\newblock {\em Probab. Theory Related Fields}, 165(1-2):1--63, 2016.

\bibitem[dRM19]{deRaynal2019}
Paul-Eric~Chaudru de~Raynal and St{\'e}phane Menozzi.
\newblock On multidimensional stable-driven stochastic differential equations
  with {B}esov drift.
\newblock {\em arXiv preprint arXiv:1907.12263}, 2019.

\bibitem[EK86]{ethier}
Stewart~N. Ethier and Thomas~G. Kurtz.
\newblock {\em Markov Processes: Characterization and Convergence}.
\newblock Wiley series in probability and mathematical statistics. Wiley, 1986.

\bibitem[FGP10]{Flandoli2010}
F.~Flandoli, M.~Gubinelli, and E.~Priola.
\newblock Well-posedness of the transport equation by stochastic perturbation.
\newblock {\em Invent. Math.}, 180(1):1--53, 2010.

\bibitem[FIR17]{Flandoli2017}
Franco Flandoli, Elena Issoglio, and Francesco Russo.
\newblock Multidimensional stochastic differential equations with
  distributional drift.
\newblock {\em Trans. Amer. Math. Soc.}, 369(3):1665--1688, 2017.

\bibitem[FRW03]{Flandoli2003}
Franco Flandoli, Francesco Russo, and Jochen Wolf.
\newblock Some {SDE}s with distributional drift. {I}. {G}eneral calculus.
\newblock {\em Osaka J. Math.}, 40(2):493--542, 2003.

\bibitem[GIP15]{Gubinelli2015Paracontrolled}
Massimiliano Gubinelli, Peter Imkeller, and Nicolas Perkowski.
\newblock Paracontrolled distributions and singular {PDE}s.
\newblock {\em Forum of Mathematics, Pi}, 3(e6), 2015.

\bibitem[GP15]{Gubinelli2015EBP}
Massimiliano Gubinelli and Nicolas Perkowski.
\newblock Lectures on singular stochastic {PDE}s.
\newblock {\em Ensaios Mat.}, 29, 2015.

\bibitem[GP17]{Gubinelli2017KPZ}
Massimiliano Gubinelli and Nicolas Perkowski.
\newblock {KPZ} reloaded.
\newblock {\em Comm. Math. Phys.}, 349(1):165--269, 2017.

\bibitem[Hai14]{Hairer2014}
Martin Hairer.
\newblock A theory of regularity structures.
\newblock {\em Invent. Math.}, 198(2):269--504, 2014.

\bibitem[HL15]{Hairer2015Simple}
Martin Hairer and Cyril Labb\'e.
\newblock A simple construction of the continuum parabolic {A}nderson model on
  {${\mathbf{R}}^2$}.
\newblock {\em Electron. Commun. Probab.}, 20:no. 43, 11, 2015.

\bibitem[HL20]{Harang2020Regularity}
Fabian~A Harang and Chengcheng Ling.
\newblock Regularity of local times associated to {V}olterra-{L\'e}vy processes
  and path-wise regularization of stochastic differential equations.
\newblock {\em arXiv preprint arXiv:2007.01093}, 2020.

\bibitem[HP20]{Harang2020Cinfinity}
Fabian~A Harang and Nicolas Perkowski.
\newblock C-infinity regularization of {ODE}s perturbed by noise.
\newblock {\em arXiv preprint arXiv:2003.05816}, 2020.

\bibitem[Jan97]{Janson1997}
Svante Janson.
\newblock {\em {Gaussian {H}ilbert spaces}}, volume 129 of {\em {Cambridge
  Tracts in Mathematics}}.
\newblock Cambridge University Press, Cambridge, 1997.

\bibitem[JS03]{shiryaev}
Jean Jacob and Albert~N. Shiryaev.
\newblock {\em Limit Theorems for Stochastic Processes}.
\newblock Grundlehren der mathematischen Wissenschaften. Springer Berlin
  Heidelberg, 2003.

\bibitem[Kin93]{Kingman1993}
J.~F.~C. Kingman.
\newblock {\em Poisson processes}, volume~3 of {\em Oxford Studies in
  Probability}.
\newblock The Clarendon Press, Oxford University Press, New York, 1993.
\newblock Oxford Science Publications.

\bibitem[KR05]{Krylov2005}
N.~V. Krylov and M.~R{\"o}ckner.
\newblock Strong solutions of stochastic equations with singular time dependent
  drift.
\newblock {\em Probab. Theory Related Fields}, 131(2):154--196, 2005.

\bibitem[KTT17]{Kusuoka2017}
Seiichiro Kusuoka, Hiroshi Takahashi, and Yozo Tamura.
\newblock Recurrence and transience properties of multi-dimensional diffusion
  processes in selfsimilar and semi-selfsimilar random environments.
\newblock {\em Electron. Commun. Probab.}, 22:Paper No. 4, 11, 2017.

\bibitem[LZ19]{Ling2019}
Chengcheng Ling and Guohuan Zhao.
\newblock Nonlocal elliptic equation in {H\"o}lder space and the martingale
  problem.
\newblock {\em arXiv preprint arXiv:1907.00588}, 2019.

\bibitem[Mat94]{Mathieu1994}
Pierre Mathieu.
\newblock Zero white noise limit through {D}irichlet forms, with application to
  diffusions in a random medium.
\newblock {\em Probab. Theory Related Fields}, 99(4):549--580, 1994.

\bibitem[MP19]{Martin2017}
J\"{o}rg Martin and Nicolas Perkowski.
\newblock Paracontrolled distributions on {B}ravais lattices and weak
  universality of the 2d parabolic {A}nderson model.
\newblock {\em Ann. Inst. Henri Poincar\'{e} Probab. Stat.}, 55(4):2058--2110,
  2019.

\bibitem[MW17]{Mourrat2017Dynamic}
Jean-Christophe Mourrat and Hendrik Weber.
\newblock The dynamic $\phi^4_3$ model comes down from infinity.
\newblock {\em Comm. Math. Phys.}, 356(3):673--753, 2017.

\bibitem[Per14]{doktorp}
Nicolas Perkowski.
\newblock {\em Studies of robustness in stochastic analysis and mathematical
  finance}.
\newblock PhD thesis, Humboldt-Universität zu Berlin,
  Mathematisch-Naturwissenschaftliche Fakultät II, 2014.

\bibitem[PZ07]{peszat_zabczyk_2007}
S.~Peszat and J.~Zabczyk.
\newblock {\em Stochastic Partial Differential Equations with Lévy Noise: An
  Evolution Equation Approach}.
\newblock Encyclopedia of Mathematics and its Applications. Cambridge
  University Press, 2007.

\bibitem[Sat99]{sato}
Ken-iti Sato.
\newblock {\em Lévy Processes and Infinitely Divisible Distributions}.
\newblock Cambridge Studies in Advanced Mathematics. Cambridge University
  Press, 1999.

\bibitem[Tan87]{Tanaka1986}
H.~Tanaka.
\newblock Limit distributions for one-dimensional diffusion processes in
  self-similar random environments.
\newblock In {\em Hydrodynamic behavior and interacting particle systems
  ({M}inneapolis, {M}inn., 1986)}, volume~9 of {\em IMA Vol. Math. Appl.},
  pages 189--210. Springer, New York, 1987.

\bibitem[Ver81]{Veretennikov1981}
Alexander~Yu. Veretennikov.
\newblock On strong solution and explicit formulas for solutions of stochastic
  integral equations.
\newblock {\em Math. USSR Sb.}, 39:387--403, 1981.

\bibitem[Zvo74]{Zvonkin1974}
A.~K. Zvonkin.
\newblock A transformation of the phase space of a diffusion process that will
  remove the drift.
\newblock {\em Mat. Sb. (N.S.)}, 93(135):129--149, 152, 1974.

\end{thebibliography}
\end{document}